\documentclass[11pt]{article}

\usepackage[a4paper,margin=2.7cm]{geometry}
\usepackage{amsmath,amssymb,amsthm,mathtools}
\usepackage{aliascnt}
\usepackage{mathrsfs}
\usepackage{bm}
\usepackage{booktabs}
\usepackage{enumitem}
\usepackage{microtype}
\usepackage{hyperref}
\usepackage[nameinlink,capitalise]{cleveref}
\hypersetup{hidelinks}
\crefname{theorem}{Theorem}{Theorems}
\Crefname{theorem}{Theorem}{Theorems}
\crefname{lemma}{Lemma}{Lemmas}
\Crefname{lemma}{Lemma}{Lemmas}
\crefname{proposition}{Proposition}{Propositions}
\Crefname{proposition}{Proposition}{Propositions}
\crefname{corollary}{Corollary}{Corollaries}
\Crefname{corollary}{Corollary}{Corollaries}
\crefname{assumption}{Assumption}{Assumptions}
\Crefname{assumption}{Assumption}{Assumptions}
\crefname{definition}{Definition}{Definitions}
\Crefname{definition}{Definition}{Definitions}
\crefname{remark}{Remark}{Remarks}
\Crefname{remark}{Remark}{Remarks}
\crefname{example}{Example}{Examples}
\Crefname{example}{Example}{Examples}
\usepackage{tikz}
\usetikzlibrary{arrows.meta,positioning,calc,fit,shapes.geometric}
\usepackage{xcolor}

\newtheorem{theorem}{Theorem}[section]
\numberwithin{equation}{section}

\newaliascnt{proposition}{theorem}
\newtheorem{proposition}[proposition]{Proposition}
\aliascntresetthe{proposition}

\newaliascnt{lemma}{theorem}
\newtheorem{lemma}[lemma]{Lemma}
\aliascntresetthe{lemma}

\newaliascnt{corollary}{theorem}
\newtheorem{corollary}[corollary]{Corollary}
\aliascntresetthe{corollary}

\newaliascnt{assumption}{theorem}
\newtheorem{assumption}[assumption]{Assumption}
\aliascntresetthe{assumption}

\newaliascnt{remark}{theorem}
\newtheorem{remark}[remark]{Remark}
\aliascntresetthe{remark}

\newaliascnt{example}{theorem}

\aliascntresetthe{example}

\theoremstyle{definition}
\newaliascnt{definition}{theorem}
\newtheorem{definition}[definition]{Definition}
\aliascntresetthe{definition}

\newcommand{\E}{\mathbb E}
\newcommand{\R}{\mathbb R}
\newcommand{\dd}{\,\mathrm d}
\newcommand{\Id}{\operatorname{id}}

\newcommand{\argmax}{\operatorname*{arg\,max}}
\newcommand{\cX}{\mathcal X}
\newcommand{\cMix}{\mathscr M}

\newcommand{\cH}{\mathcal H}
\newcommand{\cT}{\mathcal T}
\newcommand{\cE}{\mathcal E}
\newcommand{\cP}{\mathcal P}
\newcommand{\whalpha}{\widehat\alpha}
\newcommand{\tgb}{\widetilde g}
\newcommand{\bR}{\mathbf R}
\newcommand{\bS}{\mathbf S}
\newcommand{\bM}{\mathbf M}
\newcommand{\bN}{\mathbf N}
\newcommand{\bK}{\mathbf K}
\newcommand{\Rem}{\mathrm{Rem}}
\newcommand{\norm}[1]{\left\lVert #1\right\rVert}
\newcommand{\abs}[1]{\left|#1\right|}

\tikzset{
  box/.style={draw, rounded corners, align=center, inner sep=6pt, text width=3.5cm},
  smallbox/.style={draw, line width=1pt, rounded corners, align=center, inner sep=5pt, text width=2.9cm},
  arrow/.style={-{Latex[length=2.4mm]}, line width=1pt},
  dashedarrow/.style={-{Latex[length=2.4mm]}, line width=1pt, dashed},
  groupbox/.style={draw, rounded corners, dashed, inner sep=8pt}
}

\title{Quantitative Convergence of Markov Equilibria\\
for Time-Inconsistent Controlled Diffusions}
\author{Sascha Desmettre\thanks{Institute of Financial Mathematics and Applied Number Theory, Johannes Kepler University Linz, Austria. Email: \href{mailto:sascha.desmettre@jku.at}{\texttt{sascha.desmettre@jku.at}}.}}
\date{October 17, 2026}

\begin{document}
\maketitle

\begin{abstract}
Time-inconsistent stochastic control admits a natural game-theoretic interpretation in which successive decision dates are represented by different selves.  In the Markovian framework of Bj\"ork--Khapko--Murgoci, the corresponding continuous-time equilibrium HJB system can be motivated by discrete subgame-perfect games, but the general discrete-to-continuous passage is not automatic.  We establish a quantitative convergence theorem for a smooth uniformly elliptic class with strongly concave Hamiltonians and interior equilibrium controls, allowing control in both drift and diffusion.  Using the exact sample-and-hold transition law on each time step, we construct, near a given verified continuous-time BKM weak equilibrium, a unique regular discrete Markov subgame-perfect equilibrium and prove first-order convergence of continuation objects in $C^{2+\beta}$ Hölder norms with mixed regularity and of policies and values uniformly.  The key stability mechanism is structural: the discrete first-order condition makes the frozen differential feedback correction nilpotent. A two-norm dyadic Schauder argument, local Gaussian comparison, and spatial localization turn this cancellation into mesh-uniform block stability despite the second-order coupling caused by controlled volatility. 
\end{abstract}



\section{Introduction}
\label{sec:introduction}

\subsection{Motivation and the discrete-selves interpretation}

Time-inconsistent control replaces precommitment optimization by an
intrapersonal game: successive selves choose actions anticipating the
decisions of their successors. Following Strotz \cite{Strotz1955}, this
viewpoint has been developed for non-exponential discounting
\cite{EkelandLazrak2006,EkelandPirvu2008}, mean--variance preferences
\cite{BasakChabakauri2010,BjorkMurgociZhou2014}, and open-loop
linear--quadratic equilibria \cite{HuJinZhou2012}. We use the Markovian
framework of Bj\"ork--Murgoci \cite{BjorkMurgoci2014} and
Bj\"ork--Khapko--Murgoci (BKM) \cite{BKM2017}.

In discrete time, backward induction gives a natural subgame-perfect
recursion; general existence results are available under suitable
selection and compactness conditions \cite{BjorkMurgoci2014,BayraktarHan2023}.
In continuous time, equilibrium is defined through short deviations and
characterized by an extended HJB system. BKM motivate this system by a
discrete-to-continuous limit, explicitly leave the general justification
open, and note that the quadratic case was already known \cite{BKM2017}.

We study that limit for the criterion
\[
 J(t,x;\alpha)=\E_{t,x}[F(x,X_T^\alpha)]
                 +G\!\left(x,\E_{t,x}[X_T^\alpha]\right).
\]
At each grid date, a self chooses an action and holds it fixed over the next
interval. The game uses the exact diffusion transition, including controlled
volatility. Conditional on a verified smooth, uniformly elliptic continuous
equilibrium with a uniformly strongly concave Hamiltonian and interior controls,
we construct discrete equilibria and prove first-order convergence.
The limit is a BKM \emph{weak equilibrium} in the terminology of
Huang--Zhou \cite{HuangZhou2021}; convergence in a strong function-space
topology does not assert the distinct strong-equilibrium property.

\subsection{Relation to the literature and contributions}

Discrete-to-continuous equilibrium limits have important precedents.
Czichowsky \cite{Czichowsky2013} treats time-consistent mean--variance
selection in a semimartingale setting, including strong $L^2(M)$ convergence of the strategies. Bensoussan, Wong, Yam, and Yung
\cite{BensoussanEtAl2014} prove convergence for mean--variance portfolio
selection with wealth-dependent risk aversion and short-selling prohibition,
using the model's linear-in-wealth structure. Yong \cite{Yong2012} develops
partition games and equilibrium HJB equations for initial-time-dependent
criteria; his general nonlinear well-posedness and convergence analysis in Section~5 restricts the
diffusion to be control independent. Nguwi--Privault
\cite{NguwiPrivault2022} construct equilibria through weak convergence of
relaxed controls in an $n$-person approximation, also with uncontrolled
diffusion.

Related developments concern the preference specification itself.
Desmettre and Steffensen \cite{DesmettreSteffensen2023} study equilibrium
investment with random risk aversion using expected certainty equivalents.
De Gennaro Aquino, Desmettre, Havrylenko, and Steffensen develop
continuous-time equilibrium control for Kihlstrom--Mirman preferences
\cite{DeGennaroAquinoEtAl2024}, which separate risk and time preferences,
and for investment under dynamic preference uncertainty
\cite{DeGennaroAquinoEtAl2025}, where certainty equivalents are aggregated
over future preference states. These studies derive equilibrium systems
for their respective criteria. Here we study time-grid convergence for
the BKM criterion specified below; extending this analysis to those
preference frameworks requires additional arguments.

On the PDE side, Lei--Pun \cite{LeiPun2023,LeiPun2024} establish
well-posedness and Schauder stability for nonlocal fully nonlinear parabolic
systems, including the second-order coupling caused by controlled diffusion and nonlocality in both time and state \cite{LeiPun2026}.
Wang, Yu, Zhang, and Zhou \cite{WangYuZhangZhou2026} obtain equilibrium
existence by vanishing entropy regularization, with control in the drift only. Complementing these continuous-time results, we study mesh-uniform stability and quantitative convergence of discrete equilibria near a given verified regular continuous equilibrium.

The main result, \cref{thm:main}, gives $O(h)$ errors in $C^{2+\beta}$ Hölder
norms with mixed regularity for continuation quantities and in $C^0$ for policies and values. The technical core is the mesh-uniform block stability bound in \cref{prop:tube-stability}. The difficulty is that changes in the
continuation data alter the optimal action and, when volatility is
controlled, feed second-order spatial derivatives back into the recursion.
The discrete first-order condition yields a cancellation: after the
coefficients are frozen, the differential feedback correction vanishes
when applied twice. Consequently, the frozen principal evolution reduces to parabolic
propagation combined with a feedback factor that is linear in elapsed
time; see
\cref{prop:exact-linearization,prop:nilpotent}. A two-norm dyadic estimate, local Gaussian comparison, and spatial
localization extend this mechanism to the variable-coefficient
sample-and-hold recursion. An
$O(h^2)$ local residual and a first-exit argument then give the global rate.
The standard analytic inputs are parabolic Schauder smoothing
\cite{LadyzhenskayaEtAl1968,Krylov1996,Lunardi1995}, Duhamel formulas
\cite{Pazy1983}, and Littlewood--Paley estimates
\cite{BahouriCheminDanchin2011,Danchin2015}. The block estimate is specific
to the equilibrium structure; related discrete maximal-regularity questions
are studied in \cite{EvangelopoulosVeraar2026}.

\section{Continuous and discrete equilibrium problems}
\label{sec:problems}

\subsection{The continuous BKM problem}
\label{subsec:continuous-bkm}

\leavevmode Fix $0<T<\infty$. Let $(\Omega,\allowbreak\mathscr F,\allowbreak(\mathscr F_t)_{0\le t\le T},\allowbreak\mathbb P)$ support a $q$-dimensional Brownian motion $W$.  For an admissible control $a=(a_s)_{s\in[t,T]}$ with values in a compact set $A\subset\R^m$, consider the $\R^d$-valued diffusion with $\sigma(t,x,a)\in\R^{d\times q}$:
\begin{equation}
\label{eq:state-sde}
  \dd X_s=b(s,X_s,a_s)\,\dd s+\sigma(s,X_s,a_s)\,\dd W_s,
  \qquad X_t=x.
\end{equation}
Write $c=\sigma\sigma^\top$.
For a Markov control law $\alpha$, $a_s=\alpha(s,X_s)$, the objective is
\begin{equation*}
  J(t,x;\alpha)
  =\E_{t,x}\!\left[F(x,X_T^\alpha)\right]
   +G\!\left(x,\E_{t,x}[X_T^\alpha]\right).
\end{equation*}
The first argument records the current self's reference state; $G$ may also introduce nonlinear dependence on the conditional mean.

\begin{definition}[BKM weak equilibrium]
\label{def:bkm-equilibrium}
An admissible Markov control law $\whalpha$ is a \emph{BKM weak equilibrium} if, for every $(t,x)\in[0,T)\times\R^d$ and admissible local deviation $\nu$, the control
\begin{equation*}
  \alpha^{t,\varepsilon,\nu}(s,z)
  =
  \begin{cases}
    \nu(s,z), & t\le s<t+\varepsilon,\\
    \whalpha(s,z), & t+\varepsilon\le s\le T,
  \end{cases}
\end{equation*}
satisfies
\begin{equation*}
  \liminf_{\varepsilon\downarrow0}
  \frac{
  J(t,x;\whalpha)-J(t,x;\alpha^{t,\varepsilon,\nu})
  }{\varepsilon}
  \ge0.
\end{equation*}
Here $0<\varepsilon\le T-t$ is the duration of the deviation; $\nu$ is fixed as $\varepsilon\downarrow0$.
The term \emph{weak equilibrium} distinguishes this notion from stronger later notions.
\end{definition}

Let $\whalpha$ be a verified equilibrium control.  Define the continuation quantities
\begin{align}
\label{eq:def-f}
  f(t,z;y)&:=\E_{t,z}\!\left[F(y,X_T^{\whalpha})\right],\\
\label{eq:def-g}
  g(t,z)&:=\E_{t,z}\!\left[X_T^{\whalpha}\right].
\end{align}
The reference parameter $y$ is held fixed. Since $g(T,z)=z$, write
$g(t,z)=z+\tgb(t,z), \quad \tgb(T,z)=0.$
Under the standing assumptions, $\tgb$ is bounded. This decomposition separates the unbounded identity component from the bounded correction $\tgb$, to which the bounded H\"older-space estimates below apply.

For scalar functions, $D_z$ and $D_z^2$ denote the gradient and Hessian in $z$; vector-valued functions are differentiated componentwise. Here $\operatorname{tr}$ denotes the matrix trace.

For fixed $a\in A$, let
\begin{equation*}
  L_t^a\varphi(z)
  :=b(t,z,a)\cdot D_z\varphi(z)
    +\frac12 \operatorname{tr}\!\left(c(t,z,a)D_z^2\varphi(z)\right),
\end{equation*}
and set $\mathcal A^a=\partial_t+L_t^a$. The continuation equations along the equilibrium feedback are
\begin{align}
\label{eq:f-continuation-pde}
  \mathcal A^{\whalpha}f^y(t,z)&=0,
  & f^y(T,z)&=F(y,z),\\
\label{eq:g-continuation-pde}
  \mathcal A^{\whalpha}g(t,z)&=0,
  & g(T,z)&=z,
\end{align}
where $f^y(t,z):=f(t,z;y)$ and the generator acts with $y$ fixed.

The equilibrium value is
$V(t,x)=f(t,x;x)+G(x,g(t,x)).$
After this identity is substituted into the extended BKM HJB equation, the control-dependent part reduces to
\begin{equation}
\label{eq:reduced-H}
  \cH(t,x,a)
  =\left.\mathcal A^a f^x(t,z)\right|_{z=x}
   +D_2G(x,g(t,x))\cdot \mathcal A^a g(t,x).
\end{equation}
Here $D_2G(x,\mu):=\nabla_\mu G(x,\mu)$ differentiates in the second argument, with $x$ fixed.
For this equilibrium,
\begin{equation}
\label{eq:continuous-maximizer}
  \whalpha(t,x)=\argmax_{a\in A}\cH(t,x,a),
\end{equation}
and, because of \eqref{eq:f-continuation-pde}--\eqref{eq:g-continuation-pde},
$\cH(t,x,\whalpha(t,x))=0.$

\begin{remark}[The diagonal is not differentiated by the generator]
\label{rem:frozen-parameter}
The notation $\mathcal A^a f^x(t,x)$ means
$\left.\mathcal A_z^a f(t,z;y)\right|_{z=x,\,y=x}.$
Only $z$ is differentiated, not the diagonal $x\mapsto f(t,x;x)$.
\end{remark}

\subsection{Hölder spaces with mixed regularity}
\label{subsec:function-spaces}

Fix $0<\beta<1$. For a Banach space $E$ and a positive noninteger exponent $r=n+\alpha$, where $n\in\mathbb N_0$ and $0<\alpha<1$, let $C_b^r(\R^d;E)$ denote the space of bounded $E$-valued functions with bounded continuous derivatives through order $n$, whose derivatives of order $n$ are uniformly $\alpha$-Hölder continuous, measured in the norm of $E$. For scalar-valued functions, we omit $E=\R$ from the notation. We also write $C_b^0(\R^d;E)$ for bounded continuous functions with the supremum norm.

In particular, $C_b^s(\R_y^d)$ is the scalar-valued H\"older space in the reference variable $y$: writing $s=m+\gamma$, with $m\in\mathbb N_0$ and $0<\gamma<1$, its elements have bounded continuous derivatives through order $m$, and their derivatives of order $m$ are uniformly $\gamma$-H\"older continuous.

For positive noninteger $r,s$, we use \emph{Hölder spaces with mixed regularity},
\begin{equation}
\label{eq:anisotropic-F-rs}
  \cMix^{r,s}
  :=C_b^r\!\left(\R_z^d;C_b^s(\R_y^d)\right).
\end{equation}
Thus, for each state $z$, the function $y\mapsto f(z;y)$ belongs to $C_b^s(\R_y^d)$; its dependence on $z$ has regularity $r$, measured in that full function-space norm. The first index refers to the state variable $z$, and the second to the reference parameter $y$. We use the norm
\[
  \norm{f}_{\cMix^{r,s}}
  :=\bigl\|z\mapsto f(z;\cdot)\bigr\|_{C_b^r(\R_z^d;C_b^s(\R_y^d))}.
\]

For $U=(f,\tgb)$, use the product space and norm
\begin{align*}
  \cX^{2+\beta}
  &:=\cMix^{2+\beta,2+\beta}\times C_b^{2+\beta}(\R^d;\R^d),\\
  \norm{U}_{\cX^{2+\beta}}
  &:=\norm{f}_{\cMix^{2+\beta,2+\beta}}
   +\norm{\tgb}_{C_b^{2+\beta}}.
\end{align*}
For the lower state regularity used below, define analogously
\begin{equation*}
  \cX^\beta
  :=\cMix^{\beta,2+\beta}\times C_b^\beta(\R^d;\R^d),
  \qquad
  \norm{U}_{\cX^\beta}
  :=\norm{f}_{\cMix^{\beta,2+\beta}}+\norm{\tgb}_{C_b^\beta}.
\end{equation*}
Only the state regularity is reduced; regularity in the reference parameter remains $2+\beta$.

Transition operators act only in $z$. Joint $C_b^{2+\beta}$ regularity already suffices for diagonal evaluation.

\begin{lemma}[Diagonal evaluation]
\label{lem:diagonal-evaluation}
The map
\begin{equation*}
  \operatorname{Diag}:\cMix^{2+\beta,2+\beta}\longrightarrow C_b^{2+\beta}(\R^d),
  \qquad
  (\operatorname{Diag}f)(x):=f(x;x),
\end{equation*}
is bounded.  In particular, there exists $C<\infty$ such that
\begin{equation}
\label{eq:diag-bound}
  \norm{\operatorname{Diag}f}_{C_b^{2+\beta}}
  \le C\norm{f}_{\cMix^{2+\beta,2+\beta}}.
\end{equation}
\end{lemma}

\begin{proof}
Apply the chain rule twice along $(z,y)=(x,x)$; all resulting mixed derivatives and their $\beta$-H\"older seminorms are controlled by the defining norm.
\end{proof}

Use $\|u\|_{\mathcal C^s}:=\sup_{j\ge-1}2^{js}\|\Delta_j u\|_\infty$ for the inhomogeneous Besov space $\mathcal C^s:=B_{\infty,\infty}^s$. Here $\Delta_j$ are smooth dyadic Fourier blocks. For positive noninteger $s$, $\mathcal C^s=C_b^s$ with equivalent norms.

\subsection{Standing assumptions}
\label{subsec:assumptions}

The assumptions are deliberately strong: we isolate the convergence mechanism without optimizing regularity or boundary hypotheses.

\begin{assumption}[H1: control set]
\label{ass:H1}
The set $A\subset\R^m$ is nonempty, compact and convex, with $\operatorname{int}A\neq\varnothing$.  There exists an open convex neighborhood $A^\sharp\supset A$ on which all control derivatives used below are defined.
\end{assumption}

\begin{assumption}[H2: diffusion coefficients]
\label{ass:H2}
The coefficients
$b:[0,T]\times\R^d\times A^\sharp\to\R^d, \quad \sigma:[0,T]\times\R^d\times A^\sharp\to\R^{d\times q}$
are bounded and globally Lipschitz in the state variable, uniformly in $(t,a)$. Let $\beta\in(0,1)$ be the exponent fixed in Section~\ref{subsec:function-spaces}. The coefficients $b$ and $\sigma$ have continuous mixed derivatives
\[
  \partial_t^jD_x^\ell D_a^k b,
  \qquad \partial_t^jD_x^\ell D_a^k\sigma,
  \qquad 0\le j\le2,\quad 0\le\ell\le6,\quad 0\le k\le6,
\]
all uniformly bounded. For each $j$ and $k$ in these ranges, the derivatives with $\ell=6$ are uniformly $\beta$-Hölder continuous in the state variable $x$ and $\beta/2$-Hölder continuous in time. All bounds and Hölder constants are uniform in the remaining variables. These regularity orders are sufficient for the estimates below; no optimality is claimed. Six action derivatives are used for the scaled optimizer and linearization estimates in \cref{lem:scaled-tube}.
\end{assumption}

\begin{assumption}[H3: uniform ellipticity]
\label{ass:H3}
There exists $\lambda_c>0$ such that
$\xi^\top c(t,x,a)\xi\ge\lambda_c\abs{\xi}^2$
for all $(t,x,a,\xi)\in[0,T]\times\R^d\times A\times\R^d$.
\end{assumption}

\begin{assumption}[H4: payoff regularity]
\label{ass:H4}
The terminal function $F=F(y,z)$ satisfies
$(z,y)\longmapsto F(y,z)\in\cMix^{6+\beta,2+\beta}.$
The first index measures regularity in the state variable $z$, and the second in the reference parameter $y$.
The function $G=G(x,m)$ is bounded and sufficiently smooth with bounded mixed derivatives through all orders used below; the assumption
$G\in C_b^{8+\beta}(\R^d\times\R^d)$
is sufficient.
\end{assumption}

\begin{assumption}[H5: classical BKM solution]
\label{ass:H5}
There exists a classical BKM solution $(f,g,\whalpha)$ satisfying \eqref{eq:f-continuation-pde}--\eqref{eq:continuous-maximizer} and
$g=\Id+\tgb,$
with
\[
  f\in C^2\!\left([0,T];\cMix^{6+\beta,2+\beta}\right),
  \quad\text{and}\quad
  \tgb\in C^2\!\left([0,T];C_b^{6+\beta}(\R^d;\R^d)\right).
\]
All corresponding norms are uniformly bounded on $[0,T]$.
\end{assumption}

\begin{assumption}[H6: verification and admissibility]
\label{ass:H6}
The solution in \cref{ass:H5} satisfies the BKM verification hypotheses. Thus $\whalpha$ is an admissible BKM weak equilibrium as in \cref{def:bkm-equilibrium}, and the probabilistic representations \eqref{eq:def-f}--\eqref{eq:def-g} hold. The continuous equilibrium need not be unique.
\end{assumption}

\begin{assumption}[H7: uniform strong concavity]
\label{ass:H7}
The reduced Hamiltonian \eqref{eq:reduced-H} is twice continuously differentiable in $a$, and there exists $\lambda_H>0$ such that
$-D_{aa}^2\cH(t,x,a)\ge\lambda_H I_m$
for every $(t,x,a)\in[0,T]\times\R^d\times A$.
\end{assumption}

\begin{assumption}[H8: uniform interior margin]
\label{ass:H8}
There exists $\eta>0$ such that
$\operatorname{dist}(\whalpha(t,x),\partial A)\ge\eta$
for every $(t,x)\in[0,T]\times\R^d$.
\end{assumption}

\begin{assumption}[H9: uniform time grid]
\label{ass:H9}
For $N\in\mathbb N$ set
$h:=T/N, \quad t_i:=ih, \quad i=0,\dots,N.$
The main theorem concerns these uniform grids.
\end{assumption}

\subsection{Exact sample-and-hold discretization}
\label{subsec:discrete-problem}

For $a\in A$, let $P_{i,h}^a$ be the exact one-step transition operator of \eqref{eq:state-sde} with action $a$ held fixed throughout the step:
\begin{equation}
\label{eq:sample-hold-transition}
  P_{i,h}^a\varphi(x)
  :=\E\!\left[
    \varphi(X_{t_{i+1}})
    \mid X_{t_i}=x,\ a_s\equiv a\text{ on }[t_i,t_{i+1})
  \right].
\end{equation}
For a family $\varphi=\varphi(z;y)$, the operator acts only on the state variable:
$(\cP_{i,h}^a\varphi)(x;y) :=P_{i,h}^a[\varphi(\cdot;y)](x).$
Thus the time grid discretizes the \emph{game}, not the state equation: there is no Euler--Maruyama approximation in \eqref{eq:sample-hold-transition}.

Set the terminal continuation data to
\begin{align}
\label{eq:discrete-terminal-f}
  f_N^h(x;y)&:=F(y,x),\\
\label{eq:discrete-terminal-g}
  \tgb_N^h(x)&:=0.
\end{align}
Suppose $(f_{i+1}^h,\tgb_{i+1}^h)$ has already been constructed.  If player $i$ observes $X_{t_i}=x$, chooses $a\in A$, holds it on $[t_i,t_{i+1})$, and all future selves use their already constructed equilibrium strategies, then the current player's exact payoff is
\begin{equation}
\label{eq:discrete-Q}
  Q_{i,h}^h(x,a):={P}_{i,h}^a[f_{i+1}^h(\cdot;x)](x)
   +G\!\left(x,P_{i,h}^a[\Id+\tgb_{i+1}^h](x)\right).
\end{equation}
Under the standing assumptions, uniqueness of the maximizer at every backward step will be established for sufficiently small $h$ in \cref{thm:main}. Provided that the maximizer is unique, define
\begin{equation}
\label{eq:discrete-policy}
  \alpha_i^h(x):=\argmax_{a\in A}Q_{i,h}^h(x,a),
\end{equation}
and update
\begin{align}
\label{eq:discrete-f-update}
  f_i^h(x;y)
  &=P_{i,h}^{\alpha_i^h(x)}[f_{i+1}^h(\cdot;y)](x),\\
\label{eq:discrete-g-update}
  \tgb_i^h(x)
  &=P_{i,h}^{\alpha_i^h(x)}[\Id+\tgb_{i+1}^h](x)-x.
\end{align}
The associated conditional mean of the terminal state is
$g_i^h(x):=x+\tgb_i^h(x),$
and the discrete equilibrium value of player $i$ is
\begin{equation}
\label{eq:discrete-value}
  V_i^h(x):=f_i^h(x;x)+G(x,g_i^h(x)).
\end{equation}

\begin{definition}[Markov subgame-perfect equilibrium of the grid game]
\label{def:discrete-spe}
A sequence of Markov decision rules $(\alpha_0^h,\dots,\alpha_{N-1}^h)$ is a \emph{Markov subgame-perfect equilibrium} if, for every date $i$ and state $x$, $\alpha_i^h(x)$ maximizes player $i$'s payoff conditional on $X_{t_i}=x$ when all future selves $i+1,\dots,N-1$ use their prescribed rules.
\end{definition}

By the tower property, the conditional payoff in \cref{def:discrete-spe} is exactly \eqref{eq:discrete-Q}.  Hence the backward recursion \eqref{eq:discrete-policy}--\eqref{eq:discrete-g-update}, whenever well defined, is precisely the backward-induction recursion of the discrete intrapersonal game.

The state-dependent decision rule $\alpha_i^h(x)$ must be distinguished from the actually implemented continuous-time control.  Writing $X^h$ for the state process generated by the discrete policies, the sample-and-hold control on $[t_i,t_{i+1})$ is
$u_s^h=\alpha_i^h(X_{t_i}^h), \quad t_i\le s<t_{i+1}.$
It is generally different from the continuously updated feedback $\alpha_i^h(X_s^h)$.

\subsection{Main convergence theorem}
\label{subsec:main-theorem}

The construction and stability estimates below are local near a given verified continuous equilibrium. No uniqueness of the continuous BKM equilibrium is assumed. Once the grid equilibrium has been constructed, uniqueness among the Markov decision rules of \cref{def:discrete-spe} follows from backward induction and the unique global maximizer at every constructed step.

A grid equilibrium is called \emph{regular} if $(f_i^h,\tgb_i^h)\in\cX^{2+\beta}$ and $\alpha_i^h\in C_b^{2+\beta}(\R^d;\R^m)$ at all relevant grid times, for each fixed $h$.

\begin{theorem}[Quantitative convergence of discrete-time equilibria]
\label{thm:main}
Suppose \cref{ass:H1,ass:H2,ass:H3,ass:H4,ass:H5,ass:H6,ass:H7,ass:H8,ass:H9} hold. Then there exist constants $h_0>0$, $\rho>0$, and $C<\infty$ such that, for every uniform grid with $h=T/N<h_0$, the backward recursion \eqref{eq:discrete-terminal-f}--\eqref{eq:discrete-g-update} admits a unique regular Markov subgame-perfect equilibrium $(f_i^h,\tgb_i^h)_{i=0}^N$, $(\alpha_i^h)_{i=0}^{N-1}$ lying in a fixed neighborhood of the continuous solution, as expressed by
\begin{equation}
\label{eq:tube-condition}
  \max_{0\le i\le N}
  \left[
    \norm{f_i^h-f(t_i)}_{\cMix^{2+\beta,2+\beta}}
    +\norm{\tgb_i^h-\tgb(t_i)}_{C_b^{2+\beta}}
  \right]
  <\rho.
\end{equation}
The policies are unique among the Markov subgame-perfect equilibria of \cref{def:discrete-spe}.

Moreover,
\begin{align}
\label{eq:main-continuation-rate}
 &\max_{0\le i\le N}
  \left[
    \norm{f_i^h-f(t_i)}_{\cMix^{2+\beta,2+\beta}}
    +\norm{\tgb_i^h-\tgb(t_i)}_{C_b^{2+\beta}}
  \right]
  \le Ch.\\
\shortintertext{The equilibrium decision rules satisfy}
\label{eq:main-policy-rate}
 &\max_{0\le i<N}
  \norm{\alpha_i^h-\whalpha(t_i,\cdot)}_{C_b^0}
  \le Ch.\\
\shortintertext{Finally, with $V_i^h$ as in \eqref{eq:discrete-value},}
\label{eq:main-value-rate}
 &\max_{0\le i\le N}
  \norm{V_i^h-V(t_i,\cdot)}_{C_b^0}
  \le Ch.
\end{align}
The limit $\whalpha$ is the verified continuous-time BKM weak equilibrium of \cref{ass:H5,ass:H6}.
\end{theorem}

\begin{remark}[Interpretation of the convergence theorem]
\label{rem:main-interpretation}
The bound \eqref{eq:tube-condition} places the discrete continuation functions in a fixed neighborhood of the continuous solution at every grid time, with radius $\rho$ independent of $h$. This neighborhood is measured in the displayed H\"older norms, so it controls spatial derivatives as well as function values. Remaining in this neighborhood is a conclusion of the theorem, not an additional assumption on a pre-existing discrete equilibrium.
Estimate \eqref{eq:main-continuation-rate} strengthens this fixed-neighborhood bound to an $O(h)$ error for the continuation functions in those norms; \eqref{eq:main-policy-rate} and \eqref{eq:main-value-rate} give $O(h)$ errors for the decision rules and equilibrium values in the uniform norm. All bounds are uniform over grid times and states, with $C$ independent of $h$.
The limit $\whalpha$ is the continuous-time BKM weak equilibrium assumed in \cref{ass:H5,ass:H6}. The theorem constructs discrete equilibria converging to this given equilibrium; it does not establish continuous-time existence.
\end{remark}

\section{Control sensitivities and exact linearization}
\label{sec:sensitivities}

To prove \cref{thm:main}, we first need to understand how one backward equilibrium step responds to a change in the continuation functions. Since the control affects the diffusion, this response involves second spatial derivatives. We estimate the control sensitivities and derive the exact linearization of the equilibrium update. At a frozen reference point, the first-order optimality condition yields a nilpotence identity that prevents repeated feedback corrections from producing arbitrarily high differential orders. This identity is the starting point for the stability estimates in \cref{sec:block-stability}.


\subsection{Normalized control sensitivities}
\label{subsec:normalized-sensitivities}

For $0\le t\le t+h\le T$ and constant $a\in A$, let $P_{t,t+h}^a$ be the exact transition evolution of $L_s^a$. For $j=1,\ldots,6$, write
$B_j(s,a):=D_a^jL_s^a.$
For fixed control directions, each $B_j(s,a)$ is a spatial differential operator of order at most two. Under \cref{ass:H2}, for every Hölder exponent used below,
$\norm{B_j(s,a)\varphi}_{C_b^\gamma} \le C\norm{\varphi}_{C_b^{\gamma+2}}, \quad j=1,\ldots,6,$
uniformly in $(s,a)$.

For each fixed action $a$, uniform ellipticity makes the transition operator $P_{s,u}^a$ spatially smoothing: averaging over the diffusion increases regularity in the initial state. We use the following estimate to quantify this gain and its dependence on the elapsed time.
\begin{equation}
\label{eq:parabolic-smoothing}
  \norm{P_{s,u}^a\varphi}_{C_b^{\gamma+\vartheta}}
  \le C(u-s)^{-\vartheta/2}\norm{\varphi}_{C_b^\gamma},
  \qquad 0\le\vartheta<2,
\end{equation}
with constants uniform in $a\in A$ and $0\le s<u\le T$. The case $\vartheta=0$ gives boundedness on $C_b^\gamma$; positive $\vartheta$ quantifies the gain in spatial regularity.
Apply \cite[Cor.~6.1.8(i), (6.1.15), (6.1.18)]{Lunardi1995} on $X=C_b^\gamma$ with common domain $C_b^{\gamma+2}$ and interpolate. Higher spatial orders follow from differentiated Schauder estimates. Assumptions~\ref{ass:H2}--\ref{ass:H3} make the constants uniform; time reversal gives the backward evolution. No density in the full H\"older norm is required.

This smoothing compensates for the spatial derivatives introduced by $D_a^jL_s^a$ in the Duhamel representations below. Gains strictly below two derivatives yield integrable time singularities; the proof distributes the required smoothing over the intervening time intervals.

\begin{lemma}[Normalized control sensitivities]
\label{lem:S}
Let $r\in\{0,2\}$. There exist $h_0>0$ and $C<\infty$ such that, for every $0<h\le h_0$, $t+h\le T$, $a\in A$, and $k=1,\ldots,6$,
\begin{equation}
\label{eq:lemma-S-main}
  \frac1h
  \norm{D_a^kP_{t,t+h}^a\varphi}_{C_b^{r+\beta}}
  \le C\norm{\varphi}_{C_b^{r+2+\beta}}.
\end{equation}
Here the left-hand side denotes the operator norm in the $k$ control directions. Moreover, for $j=1,\ldots,5$ and $a,\widetilde a\in A$,
\begin{equation}
\label{eq:lemma-S-lipschitz-a}
  \frac1h
  \norm{D_a^jP_{t,t+h}^a-D_a^jP_{t,t+h}^{\widetilde a}}_
  {\mathcal L^j(\R^m;\mathcal L(C_b^{r+2+\beta},C_b^{r+\beta}))}
  \le C\abs{a-\widetilde a}.
\end{equation}
\end{lemma}

Lemma~\ref{lem:S} shows that control derivatives of a short-step transition are $O(h)$, even when the control affects volatility. Dividing by $h$ therefore gives uniformly bounded quantities for the scaled optimality condition and the linearization of the equilibrium update, which are the building blocks for the stability analysis below. Up to six control derivatives require only two additional spatial derivatives of the continuation function. The Lipschitz estimate is stated for $j\le5$ because its proof uses the derivative of order $j+1$, controlled here only up to order six.

The factor $h^{-1}$ is uniform even when the control enters the diffusion.
The proof in \cref{app:sensitivity} keeps the first six control derivatives
under their Duhamel integrals and distributes the spatial smoothing over
the time gaps in the ordered integrals.

\begin{corollary}[Auxiliary parameter families]
\label{cor:aux-param}
Let $\varphi=\varphi(x;y)$ belong to
$C_b^{r+2+\beta}(\R_x^d;C_b^{2+\beta}(\R_y^d)).$
Since $P_{t,s}^a$ acts only on the state variable,
$D_y^\gamma P_{t,s}^a\varphi =P_{t,s}^aD_y^\gamma\varphi.$
Consequently, all estimates of \cref{lem:S} remain valid with the scalar H\"older norms replaced by the mixed H\"older norms $\norm{\cdot}_{\cMix^{r+2+\beta,2+\beta}}$ on the input space and $\norm{\cdot}_{\cMix^{r+\beta,2+\beta}}$ on the output space, as defined in \cref{subsec:function-spaces}. Thus the estimates require two additional derivatives in the state variable, while preserving regularity in the reference parameter. In particular they apply to the full continuation family $f(\cdot; y)$.
\end{corollary}

\begin{proof}
All transition kernels and control insertions act only in $x$ and commute
with $y$-derivatives and $y$-difference quotients. Their kernel estimates extend
to $E=C_b^{2+\beta}(\R_y^d)$ by the triangle inequality for the $E$-norm.
The proof of \cref{lem:S} therefore applies with this fixed Banach range space.
\end{proof}

To apply these estimates to the second continuation component, we must also treat the unbounded identity part of $g=\Id+\tgb$. The following corollary provides the required bounds.

\begin{corollary}[Affine data]
\label{cor:affine-data}
Let $\iota(x)=x$. For $k=1,\ldots,6$ and $r\in\{0,2\}$,
\begin{equation}
\label{eq:affine-sensitivity}
  \frac1h
  \norm{D_a^kP_{t,t+h}^a\iota}_{C_b^{r+\beta}}
  \le C.
\end{equation}
Hence, for $g=\iota+\widetilde g$ with $\widetilde g\in C_b^{r+2+\beta}$,
\begin{equation}
\label{eq:g-sensitivity}
  \frac1h
  \norm{D_a^kP_{t,t+h}^ag}_{C_b^{r+\beta}}
  \le C\left(1+\norm{\widetilde g}_{C_b^{r+2+\beta}}\right).
\end{equation}
\end{corollary}

\begin{proof}
Since $L_s^a\iota=b(s,\cdot,a)$, Dynkin's formula yields
\begin{equation}
\label{eq:identity-dynkin}
  P_{t,t+h}^a\iota-\iota
  =\int_t^{t+h}P_{t,s}^ab(s,\cdot,a)\,\dd s.
\end{equation}
Differentiate \eqref{eq:identity-dynkin} up to six times in $a$. Every term contains an integral over an interval of length $h$, control derivatives of $b$, and transition sensitivities covered by \cref{lem:S}. This proves \eqref{eq:affine-sensitivity}; \eqref{eq:g-sensitivity} follows by linearity and \cref{lem:S}.
\end{proof}

\subsection{Discrete one-step objects and local stability}
\label{subsec:one-step-objects}

To compare the discrete recursion with the continuous solution, we must apply the same one-step update to different continuation functions. We therefore regard the continuation pair $U$ as an input variable: it may represent the exact continuous data, the computed discrete data, or an intermediate pair between them. This formulation lets us differentiate the update with respect to $U$ and estimate how continuation errors affect the preceding step. For continuation data $U=(f,\widetilde g)\in\cX^{2+\beta}$, define
$\Phi_{t,h}(U,a) :=\left( P_{t,t+h}^af, P_{t,t+h}^a[\Id+\widetilde g]-\Id \right),$
where the first component is understood as the family $(x,y)\mapsto P_{t,t+h}^a[f(\cdot;y)](x)$. The corresponding one-step BKM objective is
\begin{equation}
\label{eq:Q-generic}
  Q_{t,h}[U](x,a)
  :=P_{t,t+h}^a[f(\cdot;x)](x)
    +G\!\left(x,P_{t,t+h}^a[\Id+\widetilde g](x)\right).
\end{equation}
Whenever the global maximizer is uniquely defined, write
\begin{gather*}
  \alpha_{t,h}[U](x):=\argmax_{a\in A}Q_{t,h}[U](x,a)\\
  \text{and}\quad
  \cT_{t,h}(U):=\Phi_{t,h}\bigl(U,\alpha_{t,h}[U]\bigr).
\end{gather*}

Here $\Phi_{t,h}$ updates the continuation pair for a fixed action, $Q_{t,h}$ evaluates that action for the current player, and $\alpha_{t,h}[U]$ selects the maximizing action. Thus $\cT_{t,h}$ is the complete backward equilibrium step: $U_i^h=\cT_{t_i,h}(U_{i+1}^h)$.

The following quantities describe the optimality condition and its response to perturbations. Division by $h$ removes the short-time factor controlled by \cref{lem:S,cor:aux-param,cor:affine-data}:
\begin{align}
\label{eq:Gamma-def}
  \Gamma_{t,h}[U]&:=\frac1hD_aQ_{t,h}[U],\\
\label{eq:M-def}
  \bM_{t,h}[U]&:=-\frac1hD_{aa}^2Q_{t,h}[U],\\
\label{eq:R-def}
  \bR_{t,h}[U,a]&:=\frac1hD_a\Phi_{t,h}(U,a),\\
\label{eq:S-def}
  \bS_{t,h}[U,a]&:=\frac1hD_UD_aQ_{t,h}[U](\cdot,a).
\end{align}

The quantity $\Gamma$ is the scaled action gradient of the objective, $\bM$ its scaled negative action Hessian, $\bR$ the action sensitivity of the continuation update, and $\bS$ the continuation sensitivity of $\Gamma$. They will express how the optimal action and the equilibrium update change when $U$ is perturbed.

\begin{corollary}[Uniform local Lipschitz stability]
\label{cor:tube-lipschitz}
There exist $h_0>0$, $\rho_0>0$, and $C<\infty$ such that, whenever $0<h<h_0$ and $U,\widetilde U$ lie in a fixed $\rho_0$-neighborhood of the reference continuation pair $U(t+h)$ in $\cX^{2+\beta}$, the following estimates hold uniformly in $t$ and $a\in A$:
\begin{align}
\label{eq:Gamma-lipschitz}
  \norm{\Gamma_{t,h}[U]-\Gamma_{t,h}[\widetilde U]}_{C_b^\beta}
  &\le C\norm{U-\widetilde U}_{\cX^{2+\beta}},\\
\label{eq:M-lipschitz}
  \norm{\bM_{t,h}[U]-\bM_{t,h}[\widetilde U]}_{C_b^\beta}
  &\le C\norm{U-\widetilde U}_{\cX^{2+\beta}},\\
\label{eq:R-lipschitz}
  \norm{\bR_{t,h}[U,a]-\bR_{t,h}[\widetilde U,a]}_{\mathcal L(\R^m,\cX^\beta)}
  &\le C\norm{U-\widetilde U}_{\cX^{2+\beta}},\\
\label{eq:S-lipschitz}
  \norm{\bS_{t,h}[U,a]-\bS_{t,h}[\widetilde U,a]}_{\mathcal L(\cX^{2+\beta},C_b^\beta)}
  &\le C\norm{U-\widetilde U}_{\cX^{2+\beta}}.
\end{align}
If, throughout this neighborhood,
\begin{equation}
\label{eq:M-positive-tube}
  \bM_{t,h}[U](x,a)\ge\lambda_0I_m
\end{equation}
for some $\lambda_0>0$, then
\begin{equation}
\label{eq:Minv-lipschitz}
  \norm{\bM_{t,h}[U]^{-1}-\bM_{t,h}[\widetilde U]^{-1}}_{C_b^\beta}
  \le C\norm{U-\widetilde U}_{\cX^{2+\beta}},
\end{equation}
and the corresponding interior maximizers satisfy
\begin{equation}
\label{eq:alpha-lipschitz}
  \norm{\alpha_{t,h}[U]-\alpha_{t,h}[\widetilde U]}_{C_b^\beta}
  \le C\norm{U-\widetilde U}_{\cX^{2+\beta}}.
\end{equation}
\end{corollary}

These estimates show that perturbing the continuation data changes the normalized quantities in the stated norms by at most a constant times the input perturbation, uniformly as $h\to0$. Under the stated curvature and interiority conditions, the same holds for the maximizing action. This controls the ingredients of the exact linearization below; stability over many backward steps is established in \cref{sec:block-stability}.

\begin{proof}
Write
$m_{t,h}[U,a](x):=P_{t,t+h}^a[\Id+\widetilde g](x).$
Differentiating \eqref{eq:Q-generic} in $a$ gives
\begin{equation}
\label{eq:Gamma-expanded}
  \Gamma_{t,h}[U]
  =\frac1hD_aP^af^x
   +D_2G(x,m_{t,h}[U,a])\cdot
    \frac1hD_aP^a[\Id+\widetilde g].
\end{equation}
The difference between \eqref{eq:Gamma-expanded} for $U$ and $\widetilde U$ is controlled by \cref{lem:S,cor:aux-param,cor:affine-data} and boundedness of $D_{22}^2G$, proving \eqref{eq:Gamma-lipschitz}.

A second $a$-derivative yields
\begin{align*}
  \frac1hD_{aa}^2Q[U]
  ={}&\frac1hD_{aa}^2P^af^x
   +D_2G(x,m)\cdot\frac1hD_{aa}^2P^ag\\
   &+\frac1hD_{22}^2G(x,m)
     [D_aP^ag,D_aP^ag],
\end{align*}
where $g=\Id+\widetilde g$. The last term is $O(h)$ by \cref{lem:S,cor:affine-data}. The same estimates, together with bounded derivatives of $G$, prove \eqref{eq:M-lipschitz}.

Since $\Phi_{t,h}$ is affine in $U$ for fixed $a$, \cref{lem:S,cor:aux-param,cor:affine-data} imply \eqref{eq:R-lipschitz}. Differentiating \eqref{eq:Gamma-expanded} in the continuation direction $W=(W_f,W_g)$ gives
\begin{align*}
  \bS_{t,h}[U,a]W
  ={}&\frac1hD_aP^aW_f^x
  +D_2G(x,m)\cdot\frac1hD_aP^aW_g\\
  &+D_{22}^2G(x,m)
    \left[P^aW_g,\frac1hD_aP^ag\right],
\end{align*}
which proves \eqref{eq:S-lipschitz}. For the inverse matrices, use the pointwise identity
\[
 \bM[U]^{-1}-\bM[\widetilde U]^{-1}
 =\bM[U]^{-1}\bigl(\bM[\widetilde U]-\bM[U]\bigr)
   \bM[\widetilde U]^{-1},
\]
where the indices $(t,h)$ are suppressed. The lower bound \eqref{eq:M-positive-tube} gives a uniform supremum bound for the inverses. Applying the same identity at two states controls their H\"older seminorms by those of $\bM$, so their $C_b^\beta$ norms are uniformly bounded. The $C_b^\beta$ product estimate and \eqref{eq:M-lipschitz} therefore yield \eqref{eq:Minv-lipschitz}. Finally, the maximizers solve
$\Gamma_{t,h}[U]\bigl(x,\alpha_{t,h}[U](x)\bigr)=0.$
Since $D_a\Gamma=-\bM$ is uniformly invertible by \eqref{eq:M-positive-tube}, the parameter-dependent implicit-function theorem and \eqref{eq:Gamma-lipschitz}--\eqref{eq:Minv-lipschitz} yield \eqref{eq:alpha-lipschitz}.
\end{proof}

\subsection{Exact linearization of the equilibrium map}
\label{subsec:exact-linearization}

To track continuation errors through a full equilibrium step, we must account for both their direct propagation and the induced change in the maximizing action. The following proposition separates these two effects. The linearization is exact for fixed $h$: it differentiates the actual update without a short-time approximation or frozen coefficients.

\begin{proposition}[Exact one-step linearization]
\label{prop:exact-linearization}
Let $U$ lie in a neighborhood of the reference continuation functions in which the maximizer $\alpha=\alpha_{t,h}[U]$ is unique and uniformly interior and the uniform curvature bound \eqref{eq:M-positive-tube} holds. For fixed $h$, the maps $U\mapsto\alpha_{t,h}[U]$ and $U\mapsto\cT_{t,h}(U)$ are $C^1$ from this $\cX^{2+\beta}$-neighborhood into $C_b^{2+\beta}(\R^d;\R^m)$ and $\cX^{2+\beta}$, respectively. Then, for every perturbation $W\in\cX^{2+\beta}$,
\begin{align}
\label{eq:Dalpha-exact}
 &D\alpha_{t,h}[U]W
  =\bM_{t,h}[U]^{-1}\bS_{t,h}[U,\alpha]W,\\
\shortintertext{and}
\label{eq:DT-exact}
 &D\cT_{t,h}(U)W
  =P_{t,t+h}^{\alpha}W
   +h\bR_{t,h}[U,\alpha]\,
      \bM_{t,h}[U]^{-1}\,
      \bS_{t,h}[U,\alpha]W.
\end{align}
In \eqref{eq:DT-exact}, $P_{t,t+h}^{\alpha}$ acts componentwise on the $f$- and $\widetilde g$-variations, with the action $\alpha(x)$ fixed during the transition from $x$; the affine identity part of $g$ does not vary.
\end{proposition}

In \eqref{eq:Dalpha-exact}, $\bS W$ is the shift in the scaled optimality condition and $\bM^{-1}$ converts it into the action response. Thus \eqref{eq:DT-exact} separates direct propagation, $P^\alpha W$, from the feedback correction caused by reoptimizing. Integrating this derivative along the segment between two continuation pairs gives their exact update difference, hence the error recursion used in \cref{sec:main-proof}. The correction's structure is examined next and underlies the block stability in \cref{sec:block-stability}.

\begin{proof}
For fixed $h$, the zero-order bounds of Appendix~\ref{app:sensitivity} and composition give a $C^1$ score map in these spaces. Its action derivative is multiplication by $-\bM$, with bounded inverse $-\bM^{-1}$ by \eqref{eq:M-positive-tube}; the Banach-space implicit-function theorem applies.
At the optimizer,
\begin{equation}
\label{eq:FOC-general}
  D_aQ_{t,h}[U]\bigl(x,\alpha_{t,h}[U](x)\bigr)=0.
\end{equation}
Differentiate \eqref{eq:FOC-general} in the continuation direction $W$. By \eqref{eq:M-def} and \eqref{eq:S-def},
\begin{equation*}
  -h\bM_{t,h}[U]D\alpha_{t,h}[U]W
  +h\bS_{t,h}[U,\alpha]W=0,
\end{equation*}
which proves \eqref{eq:Dalpha-exact}. Next,
$\cT_{t,h}(U)=\Phi_{t,h}\bigl(U,\alpha_{t,h}[U]\bigr).$
For fixed $a$, $D_U\Phi_{t,h}(U,a)W=P_{t,t+h}^aW$, whereas $D_a\Phi_{t,h}=h\bR_{t,h}$. The chain rule together with \eqref{eq:Dalpha-exact} yields \eqref{eq:DT-exact}.
\end{proof}

\subsection{Feedback cancellation induced by discrete optimality}
\label{subsec:frozen-nilpotence}

The feedback correction in \cref{prop:exact-linearization} involves second spatial derivatives, so repeated application could appear to require ever higher regularity. We freeze the coefficients at a reference point, retaining the exact discrete optimizer in the algebraic factors. The discrete first-order condition cancels the differential feedback correction upon a second application: algebraically, $\bN_h^2=0$, so $\bN_h$ is nilpotent. Hence $(I+h\bN_h)^n=I+nh\bN_h$: repeated steps generate no higher powers of this differential correction, and parabolic smoothing controls the remaining derivatives. This cancellation provides the comparison model for the variable-coefficient block stability in \cref{sec:block-stability}.

Freeze the coefficients and reference continuation quantities at $(t_*,x_*)$ and an interior discrete optimizer $a_*$. Let
\begin{equation*}
  L_*:=b(t_*,x_*,a_*)\cdot D
      +\frac12\operatorname{tr}(c(t_*,x_*,a_*)D^2),
  \qquad
  P_h:=e^{hL_*}.
\end{equation*}
For $j=1,\dots,m$, define the constant-coefficient differential operators
\begin{equation*}
  \Lambda_j
  :=\partial_{a_j}L^a\big|_{(t_*,x_*,a_*)}
  =\partial_{a_j}b_*\cdot D
   +\frac12\operatorname{tr}((\partial_{a_j}c_*)D^2).
\end{equation*}
We write $\boldsymbol\Lambda=(\Lambda_1,\dots,\Lambda_m)$. Since all coefficients are constant after freezing,
\begin{equation*}
  [L_*,\Lambda_j]:=L_*\Lambda_j-\Lambda_jL_*=0,
  \qquad
  \frac1hD_{a_j}P_h^a\big|_{a=a_*}=P_h\Lambda_j.
\end{equation*}

Thus the two operators commute.

Let $\bar U=U(t_*+h)$ and choose $a_* = \alpha_{t_*,h}[\bar U](x_*)$ from the exact discrete objective. Retain the exact transition when freezing the algebraic coefficients:
\begin{equation*}
  m_h:=P_{t_*,t_*+h}^{a_*}[\Id+\bar{\widetilde g}](x_*),
  \qquad
  \gamma_h:=D_2G(x_*,m_h).
\end{equation*}
Here $W=(W_f,W_g)$ denotes a perturbation of the continuation pair $U$, as arises when comparing discrete and continuous continuations. The following functional combines its direct contribution through $f$ with its contribution through $g$, weighted by the derivative of $G$:

$\mathbf w_hW :=W_f(\cdot;x_*)+\gamma_h\cdot W_g.$
Here $\mathbf w_hW$ is a scalar function of the state variable; the reference parameter is frozen at $y=x_*$. Let $\bar{\bR}_h$ and $\bar{\bM}_h$ denote the values of the exact quantities \eqref{eq:R-def} and \eqref{eq:M-def} at $(t_*,x_*,\bar U,a_*)$; the $f$-component of $\bar{\bR}_h$ retains its auxiliary parameter. These algebraic coefficients are not recomputed using $P_h=e^{hL_*}$.

The discrete first-order condition implies the algebraic identity
\begin{equation}
\label{eq:wR-zero}
  \mathbf w_h\bar{\bR}_h=0.
\end{equation}
Indeed, writing $\bar g=\Id+\bar{\widetilde g}$ and $P^a=P_{t_*,t_*+h}^a$, the chain rule gives
\begin{align*}
 \mathbf w_h\bar{\bR}_h
 &=\left.\frac1h\left(D_aP^a[\bar f(\cdot;x_*)](x_*)
       +\gamma_h\cdot D_aP^a\bar g(x_*)\right)\right|_{a=a_*}\\
 &=\left.\frac1hD_a\left(P^a[\bar f(\cdot;x_*)](x_*)
       +G(x_*,P^a\bar g(x_*))\right)\right|_{a=a_*}\\
 &=\frac1hD_aQ_{t_*,h}[\bar U](x_*,a_*)=0,
\end{align*}
where the last equality is the first-order condition at the interior discrete maximizer.

The mixed sensitivity $\bS=h^{-1}D_UD_aQ$ measures how continuation perturbations change the action gradient. Define its frozen comparison model by
\begin{equation}
\label{eq:S-frozen-decomp}
  \bar{\bS}_hW
  :=P_h\boldsymbol\Lambda(\mathbf w_hW)
   +\mathbf B_hP_hW_g.
\end{equation}
The first term contains the spatial derivatives of the objective-relevant combination $\mathbf w_hW$. The second accounts for the change in the weight $D_2G$: $\mathbf B_h$ combines $D_{22}^2G$ with the frozen first control sensitivity of $g$ and introduces no spatial derivatives of $W$. Equation~\eqref{eq:S-frozen-decomp} defines a comparison model, not an exact identity for the variable-coefficient sensitivity; the comparison is justified in \cref{app:localization}.

Define the frozen differential feedback correction
\begin{equation}
\label{eq:N-def}
  \bN_h
  :=\bar{\bR}_h\bar{\bM}_h^{-1}
     \boldsymbol\Lambda\,\mathbf w_h.
\end{equation}
This constant-coefficient differential operator has order at most two.

Read from right to left, $\mathbf w_h$ combines the continuation perturbations, $\boldsymbol\Lambda$ measures their differential effect on action optimality, $\bar{\bM}_h^{-1}$ converts this into an action response, and $\bar{\bR}_h$ maps that response back to the continuations. Thus $\bN_h$ describes the differential part of the feedback loop.

\begin{proposition}[Frozen feedback cancellation]
\label{prop:nilpotent}
The operator \eqref{eq:N-def} satisfies
\begin{align}
\label{eq:N-nilpotent}
 &\bN_h^2=0.\\
\shortintertext{Moreover,}
\label{eq:P-N-commute}
 &[P_h,\bN_h]=P_h\bN_h-\bN_hP_h=0.\\
\shortintertext{Hence, with
$\mathscr E_h:=P_h(I+h\bN_h),$
we have, for every integer $n\ge0$,}
\label{eq:frozen-power}
 &\mathscr E_h^n
  =P_{nh}(I+nh\bN_h).
\end{align}
The complete frozen comparison model has the form
$\bar E_h =P_h(I+h\bN_h)+hP_h\bK_{0,h},$
where $\bK_{0,h}$ is a uniformly bounded zeroth-order operator arising from the $D_{22}^2G$ term in \eqref{eq:S-frozen-decomp}.
\end{proposition}

The identity $\bN_h^2=0$ means that a second feedback application vanishes by \eqref{eq:wR-zero}. Together with commutation, this yields \eqref{eq:frozen-power}: repeated principal steps introduce no higher powers of the differential correction. The full model $\bar E_h$ also contains the $D_{22}^2G$ contribution, but only as the bounded zeroth-order perturbation $hP_h\bK_{0,h}$. It causes no additional derivative loss and is treated perturbatively in \cref{sec:block-stability}; its accumulated effect need not vanish as $h\to0$.

\begin{proof}
Using \eqref{eq:N-def} twice and the fact that all frozen coefficients are constant,
\begin{align*}
  \bN_h^2W
  &=\bar{\bR}_h\bar{\bM}_h^{-1}
    \boldsymbol\Lambda\,
    \mathbf w_h\bar{\bR}_h\bar{\bM}_h^{-1}
    \boldsymbol\Lambda\,\mathbf w_hW\\
  &=0
\end{align*}
by \eqref{eq:wR-zero}. This proves \eqref{eq:N-nilpotent}. Since $L_*$ and every $\Lambda_j$ are constant-coefficient differential operators, they commute. The remaining factors in \eqref{eq:N-def} are frozen algebraic coefficients. Therefore \eqref{eq:P-N-commute} holds. Equations \eqref{eq:N-nilpotent} and \eqref{eq:P-N-commute} give
$\mathscr E_h^n =P_h^n(I+h\bN_h)^n =P_{nh}(I+nh\bN_h),$
which is \eqref{eq:frozen-power}.

Finally, insert the model \eqref{eq:S-frozen-decomp} and the frozen algebraic coefficients into the form of \eqref{eq:DT-exact}. The first term produces $hP_h\bN_h$, whereas the $\mathbf B_hP_hW_g$ contribution is zeroth order and can be written as $hP_h\bK_{0,h}W$. Uniform boundedness of $\bK_{0,h}$ follows from \cref{lem:S,cor:affine-data} and boundedness of $D_{22}^2G$. The comparison between this model and the actual variable-coefficient linearization is proved in \cref{app:localization}.
\end{proof}

\section{Uniform block stability}
\label{sec:block-stability}

The preceding section identifies a cancellation in a single frozen step. For \cref{thm:main}, however, errors must remain controlled through a number of backward steps that grows as $h$ decreases. We now turn that cancellation into bounds for products of the actual linearized equilibrium updates over short time blocks, with constants independent of the mesh. These bounds ensure that small local errors are not excessively amplified by the backward recursion; they provide the stability input used in \cref{sec:main-proof}.


\subsection{Frozen dyadic estimates}
\label{subsec:frozen-dyadic}

Let $\mathcal C^s=B_{\infty,\infty}^s(\R^d;E)$ be the inhomogeneous Besov space with fixed Banach range $E$. For positive noninteger $s$, it equals the H\"older space $C_b^s(\R^d;E)$ of \cref{subsec:function-spaces}, with equivalent norms. Choosing $E=C_b^{2+\beta}(\R_y^d)$ recovers $\cMix^{s,2+\beta}$ for the $f$-component; the corresponding product spaces at $s=\beta,2+\beta$ are $\cX^\beta,\cX^{2+\beta}$. Thus we change the description of regularity, not the regularity required.
A nonhomogeneous Littlewood--Paley decomposition $(\Delta_q)_{q\ge-1}$ separates a function into smooth frequency bands of size $2^q$, $q\ge0$, and a low-frequency component $\Delta_{-1}$; see \cite[Chapter~2]{BahouriCheminDanchin2011}. We use it to balance the spatial derivatives in the feedback correction against diffusion damping at each frequency scale.

Let $P_t=e^{tL_*}$ and $\mathscr E_h=P_h(I+h\bN_h)$ be the frozen operators of \cref{prop:nilpotent}. Since $L_*$ has constant coefficients, $P_t$ is Gaussian averaging followed by translation by the constant drift. Uniform ellipticity gives, for $q\ge0$ and $\lambda_q=2^{2q}$,
\[
 \|\Delta_qP_tu\|_{L^\infty(E)}
 \le Ce^{-ct\lambda_q}\|\Delta_qu\|_{L^\infty(E)}.
\]
Thus diffusion damps higher spatial frequencies more rapidly. On the same band, the operator $\bN_h$, of order at most two, contributes at most a factor $C\lambda_q$. Applying these bounds to the exact identity $\mathscr E_h^n=P_{nh}(I+nh\bN_h)$ yields
\begin{equation}
\label{eq:frozen-dyadic}
 \|\Delta_q\mathscr E_h^n u\|_{L^\infty(E)}
 \le Ce^{-cnh\lambda_q}(1+nh\lambda_q)
       \|\Delta_qu\|_{L^\infty(E)}.
\end{equation}
Since $e^{-cr}(1+r)$ is bounded for $r\ge0$, diffusion controls the additional feedback factor in \eqref{eq:frozen-dyadic}. The low-frequency block is also uniformly bounded on bounded time intervals, so
$\sup_{nh\le\Delta}\|\mathscr E_h^n\|_{\mathcal L(\mathcal C^s)}\le C_\Delta.$
The scalar heat estimate is \cite[Lemma~2.1]{Danchin2015}; a linear
change of variables and translation handle the frozen diffusion and drift.
Its $L^1$ kernel proof extends to $E$ via the integral norm inequality.
Bernstein's inequality controls the derivatives in $\bN_h$
\cite[Chapter~2]{BahouriCheminDanchin2011}.

\subsection{A two-norm block lemma}
\label{subsec:micro-duhamel}

We now describe perturbations of the frozen principal step by operators $K_{i,h}^{(2)},K_{i,h}^{(1)},K_{i,h}^{(0)}$. Put $H=\mathcal C^{2+\beta}(E)$, $J=\mathcal C^{1+\beta}(E)$, and $L=\mathcal C^\beta(E)$: these measure two, one, and zero spatial derivatives, respectively, together with $\beta$-H\"older regularity. A perturbation may be small in a weaker norm even when its strong-norm bound is not small; diffusion recovers the missing regularity. We require only the following one-step estimates, not a factorization such as $K_{i,h}=P_hB_{i,h}$:
\begin{align}
\label{eq:two-norm-classes}
 \|K_{i,h}^{(2)}u\|_H&\le b_2\|u\|_H+d_2\|u\|_0,&
 \|K_{i,h}^{(2)}u\|_L&\le h(b_2\|u\|_H+d_2\|u\|_0),\notag\\*
 \|K_{i,h}^{(1)}u\|_H&\le b_1\sqrt h\|u\|_H,&
 \|K_{i,h}^{(1)}u\|_J&\le hb_1\|u\|_H,\\*
 \|K_{i,h}^{(0)}u\|_H&\le hb_0\|u\|_H.\notag
\end{align}
Here $\|\cdot\|_0$ is the $E$-valued supremum norm, and the constants measure the sizes of the perturbations. The superscripts label estimate types, not necessarily differential orders: $K^{(2)}$ gains a factor $h$ in the norm two levels weaker, $K^{(1)}$ gains $h$ one level weaker but only $\sqrt h$ in $H$, and $K^{(0)}$ is already $O(h)$ in $H$.
Both bounds for each operator hold simultaneously; no weaker norm replaces a stronger one. For example, putting $A=b_2\|u\|_H+d_2\|u\|_0$, the two bounds for $v=K_{i,h}^{(2)}u$ imply
\[
 2^{q(2+\beta)}\|\Delta_qv\|_{L^\infty(E)}
 \le C\min\{1,h2^{2q}\}A,\qquad q\ge0.
\]
Taking the better bound on each frequency band, then combining it with diffusion damping, controls the accumulated perturbation in the strong norm $H$.
For a block of length $Mh\le\Delta\le1$, set
\[
 (\mathfrak K_h^{(k)}u)_i
 =\sum_{\ell=i}^{M-1}\mathscr E_h^{\ell-i}
                   K_{\ell,h}^{(k)}u_{\ell+1}.
\]
Each term is the perturbation created at step $\ell$, propagated back to time $i$ by the frozen principal evolution; the sum collects these contributions over the block. Lemma~\ref{lem:micro-duhamel} controls their accumulated strong norm uniformly in the number of steps. Here $\|\cdot\|_{\ell^\infty H}$ denotes the maximum $H$-norm over the block indices.

\begin{lemma}[Two-norm block estimates]
\label{lem:micro-duhamel}
Under \eqref{eq:two-norm-classes}, with
$X=\max_i\|u_i\|_H$ and $Y=\max_i\|u_i\|_0$,
\begin{align*}
 \|\mathfrak K_h^{(2)}u\|_{\ell^\infty H}&\le C(b_2X+d_2Y),\\
 \|\mathfrak K_h^{(1)}u\|_{\ell^\infty H}&\le Cb_1\sqrt\Delta X,\\
 \|\mathfrak K_h^{(0)}u\|_{\ell^\infty H}&\le Cb_0\Delta X.
\end{align*}
The constant is independent of the mesh and the number of steps.
\end{lemma}

The proof of \cref{lem:micro-duhamel} is given in \cref{app:micro-duhamel}.
Its application to the BKM linearizations additionally requires smallness of $b_2$ and control of the weaker norm $Y$, established in \cref{app:localization}.

\subsection{From frozen to path-averaged BKM linearizations}
\label{subsec:path-averaged}

We now pass from the frozen comparison model to the operators that propagate actual continuation errors. Averaging the exact derivative along the segment between the continuous reference data and perturbed data gives an exact finite-difference identity, rather than a first-order approximation.

At grid times, the index $(i,h)$ abbreviates $(t_i,h)$ in the one-step operators.
For $\bar U_i=U(t_i)$ and $\|e_i\|_{\cX^{2+\beta}}\le\rho$, define
\begin{equation}
\label{eq:path-averaged-A}
 A_{i,h}=\int_0^1D\cT_{i,h}(\bar U_{i+1}+\theta e_{i+1})\,\dd\theta.
\end{equation}
Here $e_i$ denotes a continuation perturbation at time $t_i$, and the average is over continuation data, not over time or randomness. The fundamental theorem of calculus gives
\[
 \cT_{i,h}(\bar U_{i+1}+e_{i+1})-\cT_{i,h}(\bar U_{i+1})
 =A_{i,h}e_{i+1}.
\]
Products of the $A_{i,h}$ therefore describe repeated propagation of errors. These averages need not retain the frozen cancellation. We compare them first with the linearization along the smooth reference solution, then locally with the frozen Gaussian model. Throughout, the action is held constant during each transition and substituted only afterwards, preserving the sample-and-hold rule.

\begin{proposition}[Local block stability]
\label{prop:tube-stability}
There exist $\rho_0,\Delta_0>0$ and $C_B<\infty$ such that, for the operators
in \eqref{eq:path-averaged-A} with $\|e_i\|_{\cX^{2+\beta}}\le\rho_0$,
\begin{equation}
\label{eq:tube-block-stability}
 \|A_{i,h}\cdots A_{j-1,h}\|_{\mathcal L(\cX^{2+\beta})}\le C_B,
 \qquad 0\le t_j-t_i\le\Delta_0.
\end{equation}
The constants are uniform for all sufficiently small $h$ and all block
locations.
\end{proposition}

Thus any number of steps within a sufficiently short time block amplifies a perturbation by at most $C_B$, independently of the mesh. This is the stability input in \cref{sec:main-proof}: combined with the $O(h^2)$ one-step residual proved in \cref{sec:consistency}, it yields the global $O(h)$ error.

\begin{proof}
Appendix~\ref{app:localization} supplies the technical comparisons. Control-sensitivity and optimizer estimates bound the change from perturbed data to the reference solution; Gaussian comparison and spatial localization then express the remaining difference from the frozen model in the classes \eqref{eq:two-norm-classes}. In \eqref{eq:localized-two-norm-sizes}, the leading size $b_2$ becomes small by choosing the localization radius, continuation neighborhood, and block duration small; the lower-order sizes remain bounded. Averaging along the continuation segment preserves these norm bounds even though it need not preserve nilpotence. Lemma~\ref{lem:micro-duhamel} controls the accumulated remainders. Finally, the low-norm estimate $Y\le C(\|W_M\|_H+\Delta X)$ closes the strong-norm bound by absorption, as shown at the end of Appendix~\ref{app:localization}, and proves \eqref{eq:tube-block-stability} uniformly in $h$.
\end{proof}

\section{Consistency of the discrete equilibrium recursion}
\label{sec:consistency}

Section~\ref{sec:block-stability} controls the amplification of existing continuation errors. Here we measure the new error introduced by one discrete equilibrium step even when its input is exact: inserting the continuous continuations at $t+h$ yields an optimizer error of $O(h)$ and a continuation residual of $O(h^2)$. In \cref{sec:main-proof}, the error recursion $e_i=A_{i,h}e_{i+1}+r_i^h$ combines the block stability of Section~\ref{sec:block-stability} with the residual bound proved here. Over $T/h$ steps this gives an $O(h)$ global continuation error; the decision-rule and value estimates follow. Consistency alone would not control error amplification, while stability alone would not quantify the new errors.


For $t+h\le T$ define the one-step objective obtained by inserting the exact continuous continuation values at time $t+h$:
\begin{equation}
\label{eq:Qbar-consistency}
  \overline Q_h(t,x,a)
  :=P_{t,t+h}^a[f(t+h,\cdot;x)](x)
  +G\!\left(x,P_{t,t+h}^a[\Id+\tgb(t+h,\cdot)](x)\right).
\end{equation}
This is the current player's one-step payoff when the continuations at $t+h$ are taken from the exact continuous equilibrium.
Let
$a_h^\star(t,x):=\argmax_{a\in A}\overline Q_h(t,x,a).$
The action $a_h^\star(t,x)$ is optimal for these exact continuation data and is held constant over the step. The next proposition quantifies the resulting discrepancy from the continuous equilibrium: $O(h)$ for the action and $O(h^2)$ for the continuation update.

\begin{proposition}[Second-order one-step consistency]
\label{prop:consistency}
Under \cref{ass:H1,ass:H2,ass:H3,ass:H4,ass:H5,ass:H6,ass:H7,ass:H8}, there exist $h_0>0$ and $C<\infty$ such that for $0<h<h_0$ and $t+h\le T$:
\begin{enumerate}[label=\textup{(\roman*)}]
\item $a_h^\star(t,x)$ is the unique global maximizer of \eqref{eq:Qbar-consistency}, remains uniformly in the interior of $A$, and
\begin{equation}
\label{eq:a-star-rate}
  \|a_h^\star(t,\cdot)-\whalpha(t,\cdot)\|_{C_b^{2+\beta}}
  \le Ch.
\end{equation}
\item The continuous solution has local residual
\begin{equation}
\label{eq:consistency-residual}
  \|\cT_{t,h}(U(t+h))-U(t)\|_{\cX^{2+\beta}}
  \le Ch^2.
\end{equation}

\end{enumerate}
\end{proposition}

\paragraph{Proof idea}
A short-time expansion of the exact transition expresses the action-dependent part of the one-step objective as $h\cH(t,x,a)$ plus an $O(h^2)$ remainder, with the required derivative bounds. After division by $h$, strong concavity and interiority give an $O(h)$ difference between $a_h^\star$ and $\whalpha$; differentiating the optimality condition yields the stated spatial norm bound. Finally, the continuous continuation equations cancel the first-order term at $\whalpha$. Replacing $\whalpha$ by $a_h^\star$ changes that term by only $h\,O(h)$, leaving an $O(h^2)$ continuation residual. This is the local error accumulated, under block stability, in the main $O(h)$ convergence proof. The full estimates are given in \cref{app:consistency-proof}.

\section{Proof of the main theorem}
\label{sec:main-proof}

We now combine the two ingredients established above. Section~\ref{sec:consistency} gives an $O(h^2)$ residual when the exact continuous continuations are inserted into one discrete step; Section~\ref{sec:block-stability} controls the propagation of existing errors through products of the averaged linearizations. Their combination in $e_i=A_{i,h}e_{i+1}+r_i^h$, starting from the exact terminal data $e_N=0$, yields an $O(h)$ error over the full grid. Since these estimates hold only near the continuous solution, we construct the discrete equilibrium backwards and use a first-exit argument to show that it never leaves this neighborhood. Unique one-step maximizers then give subgame perfection and uniqueness; the policy and value bounds follow from the continuation estimate.

\begin{proof}[Proof of \cref{thm:main}]
The errors below are defined whenever the corresponding discrete continuations have been constructed.
Write $\bar U_i=U(t_i)$, $e_i=U_i^h-\bar U_i$, and
$\|\cdot\|_{\cX}=\|\cdot\|_{\cX^{2+\beta}}$.

\paragraph{Step 1: well-defined updates in a fixed neighborhood}
We first ensure that nearby continuation data yield a unique interior action, so that each backward update is defined.
By \cref{prop:consistency,cor:tube-lipschitz}, for small $h$ and
$\|V-\bar U_{i+1}\|_{\cX}\le\rho_{\rm s}$,
\[
 \bM_{i,h}[V](x,a)\ge\tfrac12\lambda_H I_m,\qquad
 \|\Gamma_{i,h}[V](\cdot,\whalpha(t_i,\cdot))\|_\infty
       \le C(h+\rho_{\rm s}).
\]
Thus the objective is strongly concave on $A$. Its unique maximizer
satisfies $|\alpha_{i,h}[V](x)-\whalpha(t_i,x)|\le C(h+\rho_{\rm s})$
by strong concavity and the variational inequality for a constrained
maximizer. Taking $h,\rho_{\rm s}$ small preserves an interior margin
$\eta/2$. The update is therefore defined throughout this neighborhood, and
\eqref{eq:alpha-lipschitz} controls its policy map.

\paragraph{Step 2: consistency residual and stable error propagation}
We now combine the residual from Section~\ref{sec:consistency} with the averaged linearizations of Section~\ref{sec:block-stability}, using blocks of uniformly bounded duration.
Let $\rho_0,\Delta_0,C_B$ be as in \cref{prop:tube-stability}, with
$C_B\ge1$, and fix $0<\rho<\frac12\min\{\rho_{\rm s},\rho_0\}$.
For $h\le\Delta_0/2$, partition the grid into blocks
$0=n_0<\cdots<n_J=N$ of at most $\lfloor\Delta_0/h\rfloor$ steps,
with $J\le J_0$ independently of $h$.
The residual $r_i^h=\cT_{i,h}(\bar U_{i+1})-\bar U_i$ satisfies
$\|r_i^h\|_{\cX}\le C_\tau h^2$ by \cref{prop:consistency}.

Whenever $\|e_{i+1}\|_{\cX}<\rho$, the next value
$U_i^h=\cT_{i,h}(U_{i+1}^h)$ exists and the mean-value formula gives
\begin{equation}
\label{eq:main-error-recursion}
 e_i=A_{i,h}e_{i+1}+r_i^h,\qquad
 A_{i,h}=\int_0^1D\cT_{i,h}(\bar U_{i+1}+\theta e_{i+1})\,\dd\theta.
\end{equation}
For $i\le\ell$ put $\cE_{i,\ell}^h=A_{i,h}\cdots A_{\ell-1,h}$,
with $\cE_{i,i}^h=I$. On a block whose input errors remain in the neighborhood,
\cref{prop:tube-stability} gives $\|\cE_{i,\ell}^h\|\le C_B$, hence
\begin{equation}
\label{eq:main-variation-constants}
 e_i=\cE_{i,m}^he_m+\sum_{\ell=i}^{m-1}\cE_{i,\ell}^hr_\ell^h.
\end{equation}

\paragraph{Step 3: backward construction and exclusion of an exit}
The terminal continuations agree exactly. We construct earlier steps successively and show that the stability and residual bounds prevent the error from reaching the neighborhood boundary.
Set $C_J=0$ and $C_k=C_B(C_{k+1}+C_\tau\Delta_0)$ for $k<J$.
All these constants are bounded by
$C_*=C_\tau\Delta_0\sum_{\ell=1}^{J_0}C_B^\ell$.
Choose $h_0$ so that $C_*h_0<\rho/2$ and all preceding estimates apply.
Starting from $e_N=0$, suppose by backward induction that the right endpoint
of block $k$ satisfies $\|e_{n_{k+1}}\|_{\cX}\le C_{k+1}h$.
If $j$ were the first backward exit in that block, then $U_j^h$ would
still be defined, since its input $U_{j+1}^h$ is inside the neighborhood.
All operators between $j$ and $n_{k+1}$ satisfy the block bound,
so \eqref{eq:main-variation-constants} gives
\[
 \|e_j\|_{\cX}
 \le C_B\bigl(C_{k+1}h+C_\tau(n_{k+1}-j)h^2\bigr)
 \le C_kh\le C_*h<\rho/2,
\]
contradicting an exit. The same estimate holds at every index of the block.
Induction constructs the equilibrium on the entire grid and proves
\begin{equation}
\label{eq:continuation-rate-proof}
 \max_{0\le i\le N}\|U_i^h-\bar U_i\|_{\cX}\le C_*h<\rho.
\end{equation}

\paragraph{Step 4: equilibrium, uniqueness, and policy and value convergence}
Having constructed the continuations on the full grid, we verify the game-theoretic claims and transfer their error bound to actions and values.
By \eqref{eq:discrete-Q}, each constructed policy is the unique best
response to the later policies, so the recursion is subgame perfect.
Any other equilibrium of \cref{def:discrete-spe} has the same terminal
data and hence the same unique maximizer at $N-1$. Backward induction
identifies its policies and continuation values at every date.

To transfer the continuation error to the policy, introduce the intermediate
policy $a_i^\star:=\alpha_{i,h}[\bar U_{i+1}]$, obtained by solving the discrete
optimization problem with the exact continuous continuation data. The triangle
inequality separates the effect of the continuation error from the error of this
single discrete optimization step:
\begin{align*}
 \|\alpha_i^h-\whalpha(t_i)\|_{C_b^0}
 &\le \|\alpha_{i,h}[U_{i+1}^h]-\alpha_{i,h}[\bar U_{i+1}]\|_{C_b^0}
       +\|a_i^\star-\whalpha(t_i)\|_{C_b^0}\\
 &\le C\|e_{i+1}\|_{\cX}+Ch\le Ch.
\end{align*}
Here \cref{cor:tube-lipschitz} controls the first term by Lipschitz continuity
of the optimizer, while \cref{prop:consistency} controls the second term;
\eqref{eq:continuation-rate-proof} then gives the final bound.

For the value function, recall that
$V_i^h(x)=f_i^h(x;x)+G(x,x+\tgb_i^h(x))$, with the analogous formula
for $V(t_i,x)$. The diagonal bound \eqref{eq:diag-bound} controls the
$f$-difference, and the mean value theorem with bounded $D_2G$ controls the
$G$-difference. Since the identity part of $g$ cancels, this yields
\[
 \|V_i^h-V(t_i)\|_{C_b^0}
 \le C\|f_i^h-f(t_i)\|_{\cMix^{2+\beta,2+\beta}}
      +C\|\tgb_i^h-\tgb(t_i)\|_{C_b^0}
 \le C\|e_i\|_{\cX}\le Ch.
\]
These are the remaining claims of \cref{thm:main}.
\end{proof}

\appendix

\section{Proof of the normalized control sensitivities}
\label{app:sensitivity}

We prove \cref{lem:S} using the uniform parabolic estimate
\eqref{eq:parabolic-smoothing}. All constants below depend only on the
coefficient bounds, ellipticity, $T$, $A$, and $\beta$, and are independent of
$a,t,h$. Put $B_j(s,a)=D_a^jL_s^a$, $1\le j\le6$. For control directions of
norm at most one, these are uniformly bounded operators from
$C_b^{\gamma+2}$ to $C_b^\gamma$ at all exponents used below.

\paragraph{Duhamel representation}
Duhamel's formula expresses action derivatives of the transition through derivatives of its generator; parabolic smoothing then controls their short-time size.

For $1\le\ell\le6$ and second-order differential insertions $C_1,\ldots,C_\ell$,
write
\begin{equation}
\label{eq:app-insertion}
 I_\ell(C_1,\ldots,C_\ell)
 :=\int_{t<s_1<\cdots<s_\ell<t+h}
 P_{t,s_1}^a C_1(s_1)P_{s_1,s_2}^a\cdots
 C_\ell(s_\ell)P_{s_\ell,t+h}^a\,\dd s_1\cdots\dd s_\ell.
\end{equation}
The time gaps are the intervals between consecutive integration times, together with the outer intervals $(t,s_1)$ and $(s_\ell,t+h)$; we distribute the required smoothing over these gaps.

For smooth bounded terminal data $\varphi$, put $u^a(s)=P_{s,t+h}^a\varphi$.
The evolution equation is the backward Kolmogorov equation
\[
 \partial_su^a+L_s^au^a=0,\qquad u^a(t+h)=\varphi.
\]
Differentiating in a fixed control direction $\eta$ gives, for $v(s)=D_au^a(s)[\eta]$,
\[
 \partial_sv+L_s^av=-B_1(s,a)[\eta]u^a(s),\qquad v(t+h)=0.
\]
Differentiating the generator produces $B_1=D_aL^a$; variation of constants gives
\begin{equation}
\label{eq:app-D1}
 D_aP_{t,t+h}^a[\eta]
 =\int_t^{t+h}P_{t,s}^a B_1(s,a)[\eta]P_{s,t+h}^a\,\dd s.
\end{equation}
The notation $\eta^{\otimes j}$ means that the same direction $\eta$ is inserted into all $j$ arguments of $D_a^jP^a$ or $B_j$.
Repeated differentiation yields, for every $k\le6$,
\begin{equation}
\label{eq:app-Dk}
 D_a^kP_{t,t+h}^a[\eta^{\otimes k}]
 =\sum_{\ell=1}^k\ \sum_{\substack{j_1+\cdots+j_\ell=k\\j_\nu\ge1}}
 \frac{k!}{j_1!\cdots j_\ell!}
 I_\ell\bigl(B_{j_1}[\eta^{\otimes j_1}],\ldots,
 B_{j_\ell}[\eta^{\otimes j_\ell}]\bigr),\qquad k\le6.
\end{equation}
Each differentiation either acts on an existing insertion or inserts $B_1$
into one of the evolution factors. This proves \eqref{eq:app-Dk} by induction.
Since these derivatives are symmetric multilinear maps, polarization recovers mixed directions from linear combinations of repeated-direction evaluations at sums of directions.
We first derive these identities for smooth bounded data; see \cite[Sec.~6.2, Prop.~6.2.2 and Cor.~6.2.4]{Lunardi1995} for variation of constants.

\paragraph{Passage to the full H\"older spaces}
We extend the identities to every function in the stated H\"older spaces, where smooth functions need not be dense in the original H\"older norm.
For $\varphi\in C_b^{r+2+\beta}$, let $\varphi_n=\rho_n*\varphi$ be its standard smooth averages.
Their original H\"older norms remain uniformly bounded, and $\varphi_n\to\varphi$ in $C_b^{r+2+\beta'}$ for every $0<\beta'<\beta$.
The time-ordered integral estimates below at exponent $\beta'$ give convergence of the insertion integrals, and passage to the limit in the action-variable fundamental-theorem-of-calculus identities identifies the pointwise control derivatives through order six.
The bounds at exponent $\beta$ then pass to the limit by convergence of spatial derivatives and the uniform bounds on their H\"older difference quotients.
The zero-order estimates follow identically from uniformly bounded mollifications in $C_b^s$ converging in $C_b^{s-\delta}$, where $s\in\{\beta,1+\beta,2+\beta\}$ and $0<\delta<\beta$.

\paragraph{Estimate of the time-ordered integrals}
Fix $r\in\{0,2\}$ and $0<\varepsilon<\beta/6$. Starting from the right in
\eqref{eq:app-insertion}, the first insertion maps $C_b^{r+2+\beta}$ to
$C_b^{r+\beta}$. Across each of the $\ell-1$ interior gaps recover
$2-\varepsilon$ derivatives before applying the next insertion. The remaining
$(\ell-1)\varepsilon$ derivatives are recovered across the outermost gap.
The lowest intermediate exponent is $r+\beta-(\ell-1)\varepsilon>0$.
Each interior smoothing step contributes $(s_{j+1}-s_j)^{-1+\varepsilon/2}$, whereas the outermost step contributes $(s_1-t)^{-(\ell-1)\varepsilon/2}$.
The rightmost transition is bounded without smoothing; composing these estimates with the bounds $C_j:C_b^{\gamma+2}\to C_b^\gamma$ yields the integrand below.
Hence \eqref{eq:parabolic-smoothing} gives
\begin{align}
\label{eq:app-simplex-bound}
 \|I_\ell(C_1,\ldots,C_\ell)\varphi\|_{C_b^{r+\beta}}
 &\le C\|\varphi\|_{C_b^{r+2+\beta}}
 \int_{t<s_1<\cdots<s_\ell<t+h}
 (s_1-t)^{-(\ell-1)\varepsilon/2}\notag\\
 &\quad\times
 \prod_{j=1}^{\ell-1}(s_{j+1}-s_j)^{-1+\varepsilon/2}
 \,\dd s_1\cdots\dd s_\ell
 \le Ch\|\varphi\|_{C_b^{r+2+\beta}}.
\end{align}
For $\ell=1$, the product is empty and the integral equals $h$.
For $2\le\ell\le6$, every singular exponent is less than one and their sum is $\ell-1$.
The substitution $s_j=t+hu_j$ leaves a finite integral over $0<u_1<\cdots<u_\ell<1$, multiplied by $h^\ell h^{-(\ell-1)}=h$.
For two insertions, the smoothing gains on the interior and leftmost gaps are $(2-\varepsilon)+\varepsilon$; for three insertions they are $(2-\varepsilon)+(2-\varepsilon)+2\varepsilon$.
This allocation avoids recovering two full derivatives on an interior gap, which would produce a nonintegrable factor $(s_{j+1}-s_j)^{-1}$.

Equations \eqref{eq:app-Dk} and \eqref{eq:app-simplex-bound} imply
\eqref{eq:lemma-S-main} after division by $h$. For $j=1,\ldots,5$, the identity
\[
 D_a^jP^a-D_a^jP^{\widetilde a}
 =\int_0^1D_a^{j+1}P^{\widetilde a+\theta(a-\widetilde a)}
       [a-\widetilde a]\,\dd\theta
\]
and convexity of $A$ give \eqref{eq:lemma-S-lipschitz-a}. The auxiliary-parameter
and affine-data consequences are proved with their statements in
\cref{subsec:normalized-sensitivities}.

\paragraph{Zero-order bounds and the scaled estimate}
Using the representation \eqref{eq:app-Dk}, we now estimate the operators on $C_b^s$ without additional regularity of the input:
\begin{equation}
\label{eq:action-zero-order}
 \|D_a^kP_{t,t+h}^a\|_{\mathcal L(C_b^s)}\le C,
 \qquad 1\le k\le6,\quad s=\beta,1+\beta,2+\beta.
\end{equation}
For a term with $\ell$ insertions, recover $2-\varepsilon$ derivatives across the rightmost time gap $(s_\ell,t+h)$ as well as each interior gap $(s_j,s_{j+1})$.
Each subsequent insertion loses two derivatives, leaving regularity $s-\ell\varepsilon>0$ when $0<\varepsilon<\beta/6$.
Recover the missing $\ell\varepsilon$ derivatives across the leftmost gap $(t,s_1)$.
All $\ell+1$ gap factors are integrable: their singular exponents are less than one and sum to $\ell$.
Rescaling $s_j=t+hu_j$ therefore gives $h^\ell h^{-\ell}=1$, explaining why this bound has no factor $h$.

For $u\in C_b^{2+\beta}$, put $v=D_a^kP_{t,t+h}^a u$.
Lemma~\ref{lem:S} and \eqref{eq:action-zero-order} with $s=2+\beta$ give, respectively,
\[
 \|v\|_{C_b^\beta}\le Ch\|u\|_{C_b^{2+\beta}},
 \qquad
 \|v\|_{C_b^{2+\beta}}\le C\|u\|_{C_b^{2+\beta}}.
\]
The interpolation inequality between these two output norms yields
\[
 \|v\|_{C_b^{1+\beta}}
 \le C\|v\|_{C_b^\beta}^{1/2}\|v\|_{C_b^{2+\beta}}^{1/2}
 \le C\sqrt h\|u\|_{C_b^{2+\beta}}.
\]
To combine these bounds, define, for scalar, vector, or operator-valued $v$,
\begin{equation}
\label{eq:scaled-holder-norm}
 \|v\|_{\beta,h}:=\|v\|_{C_b^\beta}
       +\sqrt h\|v\|_{C_b^{1+\beta}}+h\|v\|_{C_b^{2+\beta}}.
\end{equation}
Taking the weighted sum of these three bounds gives
\[
 \|D_a^kP_{t,t+h}^a u\|_{\beta,h}
 \le Ch\|u\|_{C_b^{2+\beta}},\qquad 1\le k\le6.
\]
This estimate will be used in \cref{lem:scaled-tube}.

\section{Proof of the two-norm block estimate}
\label{app:micro-duhamel}

\begin{proof}[Proof of \cref{lem:micro-duhamel}]
Set $R=b_2X+d_2Y$ and, for $q\ge0$, $\lambda_q=2^{2q}$.
We weight each frequency block by $2^{q(2+\beta)}$, as in the definition of the $H$ norm.
The two bounds for $K^{(2)}$ imply
\[
 2^{(2+\beta)q}\|\Delta_qK_{\ell,h}^{(2)}u_{\ell+1}\|_\infty
 \le CR\min\{1,h\lambda_q\}.
\]
The frozen evolution commutes with $\Delta_q$, so \eqref{eq:frozen-dyadic} applies to each output block of the perturbation.
Writing $n=\ell-i$ for the number of propagation steps and extending the sum to $M-1$, the weighted $q$-block of the convolution is bounded by
\[
 CR\min\{1,v\}\sum_{n=0}^{M-1}(1+nv)e^{-cnv}\le C R,
 \qquad v=h\lambda_q.
\]
For $v\le1$, exponential damping becomes effective after order $v^{-1}$ steps; summing the geometric series and its differentiated series gives $O(v^{-1})$.
For $v\ge1$, both sums are uniformly bounded, giving $O(1)$.
This includes the unsmoothed insertion $n=0$ and makes no assumption about
frequency preservation by $K_{\ell,h}^{(2)}$.

For $K^{(1)}$ the corresponding factor is
$b_1X\min\{\sqrt h,h\sqrt{\lambda_q}\}$.
If $h\lambda_q\ge1$, its convolution is at most $Cb_1\sqrt h X\le Cb_1\sqrt\Delta X$, since $h\le\Delta$.
If $h\lambda_q<1$, it is bounded by
\[
 Cb_1Xh\sqrt{\lambda_q}\min\{M,(h\lambda_q)^{-1}\}
 \le Cb_1X\min\{\Delta\sqrt{\lambda_q},\lambda_q^{-1/2}\}
 \le Cb_1\sqrt\Delta X.
\]
The last inequality follows from $\min\{a,b\}\le\sqrt{ab}$, since the two terms have product $\Delta$.
For $K^{(0)}$, uniform boundedness of the frozen powers gives $Cb_0MhX\le Cb_0\Delta X$.
For $q=-1$, the weaker output norms control the low-frequency block directly, giving a factor $h$ per insertion.
Bounded frozen propagation and $Mh\le\Delta$ therefore give contributions at most $C\Delta R$,
$Cb_1\Delta X$, and $Cb_0\Delta X$, respectively. Since $\Delta\le1$,
these obey the asserted bounds. The supremum over $q$ completes the proof.
\end{proof}

\section{The sample-and-hold comparison and localization}
\label{app:localization}

We use the product-space convention of \cref{subsec:function-spaces}, extended to the intermediate and supremum norms by
\[
 \cX^s:=\cMix^{s,2+\beta}\times C_b^s(\R^d;\R^d),
 \qquad s\in\{0,\beta,1+\beta,2+\beta\},
\]
with the sum norm. No identification of $C_b^0$ with the order-zero Besov space is used. Thus
$H=\cX^{2+\beta}$, $J=\cX^{1+\beta}$, $L=\cX^\beta$, and $Y_0=\cX^0$.
The fixed Banach range in the notation of Section~\ref{sec:block-stability} is
$E=C_b^{2+\beta}(\R_y^d)\times\R^d$. In particular,
$\|U\|_0=\|U\|_{Y_0}$ retains the reference-parameter regularity of the $f$-component. Constants may depend on the reference solution and the standing bounds, but never on $h$.

\paragraph{Scaled norms and the optimizer}
We use the scaled norm \eqref{eq:scaled-holder-norm} introduced in \cref{app:sensitivity}.
The $C_b^\beta$ algebra inequality and the product rule give a uniform
product bound, weighting $D v D w$ as $(\sqrt h Dv)(\sqrt h Dw)$.
The chain rule likewise gives uniform bounds for $v(x,a(x))$ when
$\|a\|_{\beta,h}$ and the first three action derivatives of $v$ in this
norm are bounded; the third controls H\"older differences. For differences
of compositions, apply the same rule to $v_a$ and integrate between actions.

\begin{lemma}[Scaled estimates near the continuous solution]
\label{lem:scaled-tube}
For $U,V$ in a sufficiently small fixed $H$-neighborhood of the reference continuation pair $U(t+h)$,
\begin{gather}
\label{eq:scaled-optimizer}
 \|\alpha[U]\|_{\beta,h}\le C,\qquad
 \|\alpha[U]-\alpha[V]\|_{\beta,h}\le C\|U-V\|_H,\\
\label{eq:two-norm-tube}
 \begin{aligned}
 \|(D\cT_h(U)-D\cT_h(V))W\|_H
     &\le C\|U-V\|_H\|W\|_H,\\
 \|(D\cT_h(U)-D\cT_h(V))W\|_L
     &\le Ch\|U-V\|_H\|W\|_H.
 \end{aligned}
\end{gather}
The reference action $a_i=\alpha_{i,h}[\bar U_{i+1}]$ and the reference
coefficients $\bR_i,\bM_i^{-1},\gamma_i=D_2G(x,P^{a_i}g)$ have uniformly
bounded, unscaled $C_b^{2+\beta}$ norms. Their supremum-norm time moduli are
bounded by $C|t_i-t_j|^{\beta/2}$, with the auxiliary range retained.
\end{lemma}

\begin{proof}
First we justify uniform curvature and interiority throughout the neighborhood.
For the reference data $\bar U=U(t+h)$, the consistency expansion gives
$\bM[\bar U](x,a)\ge(\lambda_H/2)I_m$ for all $a\in A$, and its maximizer $a_{\rm ref}(x)$ has distance at least $\eta/2$ from $\partial A$; see \eqref{eq:qhat-concavity} and the interiority argument in the proof of \cref{prop:consistency}. The argument in Appendix~\ref{app:consistency-proof} is independent of block stability.
By \cref{cor:tube-lipschitz}, if $\|U-\bar U\|_H\le\rho$ and $\rho$ is sufficiently small, then
$\bM[U](x,a)\ge\lambda_0I_m,\quad \lambda_0:=\lambda_H/4,$
uniformly in $x,a$ and small $h$.
Hence the one-step objective has a unique global maximizer $a_U(x)$ on the compact convex set $A$.
The variational inequality for this maximizer, strong concavity, and $\Gamma[\bar U](x,a_{\rm ref})=0$ give
\[
 \lambda_0|a_U-a_{\rm ref}|
 \le |\Gamma[U](x,a_{\rm ref})|
 \le C\|U-\bar U\|_H.
\]
This argument does not assume that $a_U$ is interior.
After decreasing $\rho$ so that $C\rho/\lambda_0\le\eta/4$, we obtain
$\operatorname{dist}(a_U(x),\partial A)\ge\eta/4$.
Thus both the inverse-Hessian bound and the ordinary first-order condition used below hold on one fixed neighborhood, independently of $h$.

The estimates in \cref{app:sensitivity} imply, for $1\le k\le6$,
\begin{equation}
\label{eq:scaled-transition}
 \|D_a^kP_{t,t+h}^a W\|_{\beta,h}\le Ch\|W\|_H,
 \qquad
 \|D_a^kP_{t,t+h}^a\|_{\mathcal L(C_b^s)}\le C,
 \quad s=\beta,1+\beta,2+\beta.
\end{equation}
These bounds also apply to families and to the recentered affine data.
For compositions with $G$, write $m=x+\widetilde m$ and use the bounded
derivatives of $(x,v)\mapsto G(x,x+v)$; no boundedness of $m$ itself is
needed.
Expand $\Gamma,\bM,\bR,\bS$ as in
\cref{subsec:one-step-objects}. Division by $h$ in each action-sensitive
factor is canceled by \eqref{eq:scaled-transition}. Product and composition
estimates therefore bound their scaled norms, and their first variations
in $U$ and $a$, uniformly; a variation $W$ contributes $C\|W\|_H$.
To see that all derivative orders are covered, the most demanding factor
is $h^{-1}D_a^3Q$: its three further action derivatives and two spatial
derivatives require at most six action derivatives of the transition
and eight derivatives of $G$. These are included in
\cref{ass:H2,ass:H4}.

The unscaled $C_b^\beta$ bound for $\alpha$ follows from uniform concavity
and the score bounds. The scaled bounds for the score and its action derivatives give
$\|\Gamma_x\|_{C_b^\beta}\le Ch^{-1/2}$,
$\|\Gamma_{xx}\|_{C_b^\beta}\le Ch^{-1}$,
$\|\Gamma_{xa}\|_{C_b^\beta}\le Ch^{-1/2}$, and
$\|\Gamma_{aa}\|_{C_b^\beta}\le C$, including evaluation at $a=\alpha(x)$.
Differentiating $\Gamma(x,\alpha(x))=0$ once and twice gives
\[
 D\alpha=\bM^{-1}\Gamma_x,\qquad
 D^2\alpha=\bM^{-1}
   \{\Gamma_{xx}+2\Gamma_{xa}D\alpha
                     +\Gamma_{aa}[D\alpha,D\alpha]\}.
\]
The first identity gives $\|D\alpha\|_{C_b^\beta}\le Ch^{-1/2}$.
Each term in braces in the second identity is then bounded by $Ch^{-1}$, so
$\|D^2\alpha\|_{C_b^\beta}\le Ch^{-1}$.
The weights $\sqrt h$ and $h$ compensate for precisely these possible losses and prove the first bound in \eqref{eq:scaled-optimizer}.
Inverses obey the same scaled product bounds since $\bM\ge\lambda_0I$.
For the second bound, integrate $D_U\alpha[U]W=\bM^{-1}\bS W$, whose scaled
norm is at most $C\|W\|_H$, along $V+\theta(U-V)$.

For the transition part of $D\cT_h$, use the exact, pointwise identity
\[
 (P^{\alpha[U]}-P^{\alpha[V]})W
 =\int_0^1
   (D_aP^{a_\theta}W)[\alpha[U]-\alpha[V]]\,\dd\theta,
 \quad a_\theta=\alpha[V]+\theta(\alpha[U]-\alpha[V]).
\]
By \eqref{eq:scaled-transition} and the scaled composition estimate,
$\|D_aP^{a_\theta}W\|_{\beta,h}\le Ch\|W\|_H$.
Together with \eqref{eq:scaled-optimizer}, the scaled product inequality therefore bounds the right side by $Ch\|U-V\|_H\|W\|_H$.
The identity is evaluated pointwise after each fixed-action transition; no state-dependent action is moved inside a transition operator.

For the feedback term, abbreviate by $\bR_U,\bM_U,\bS_U$ the quantities at continuation data $U$ and action $\alpha[U]$, and similarly for $V$.
The exact difference is
\begin{align*}
 h(\bR_U\bM_U^{-1}\bS_U-\bR_V\bM_V^{-1}\bS_V)W
 ={}&h(\bR_U-\bR_V)\bM_U^{-1}\bS_UW\\*
 &+h\bR_V(\bM_U^{-1}-\bM_V^{-1})\bS_UW\\*
 &+h\bR_V\bM_V^{-1}(\bS_U-\bS_V)W.
\end{align*}
By the first-variation estimates and \eqref{eq:scaled-optimizer}, each factor difference is bounded by $C\|U-V\|_H$ in the scaled norm, including the induced action change.
The product estimate therefore bounds each term by $Ch\|U-V\|_H\|W\|_H$.
Combining the transition and feedback contributions yields
\[
 \|(D\cT_h(U)-D\cT_h(V))W\|_{\beta,h}
 \le Ch\|U-V\|_H\|W\|_H.
\]
The $L$-component retains the factor $h$; dividing the weighted $H$-component by $h$ gives the other bound in \eqref{eq:two-norm-tube}.

Along the smooth reference solution, \cref{lem:S} with $r=2$ gives the
unscaled coefficient bounds instead. Differentiating the Duhamel
integrals and reference data in time at fixed step length retains the
factor $h$, and gives the asserted time modulus. Diagonal evaluation uses
the available two auxiliary derivatives throughout.
\end{proof}

\paragraph{A Gaussian comparison with the action still fixed}
For a constant action $v$, put
\[
 p(t,x,v,\xi)=i b(t,x,v)\cdot\xi
                  -\tfrac12\xi^\top c(t,x,v)\xi,
 \qquad \Pi_{t,h}^v=\operatorname{Op}(e^{hp(t,x,v,\xi)}).
\]
Here $\operatorname{Op}$ denotes left quantization: the coefficients are
evaluated at the starting point $x$. Equivalently, $\Pi_{t,h}^vu(x)$ is
the expectation of $u(x+hb(t,x,v)+\sqrt h\,\sigma(t,x,v)Z)$ for a standard
Gaussian vector $Z$ of the appropriate dimension. Both transitions hold the action fixed; the Gaussian comparison additionally freezes the coefficients at $(t,x)$. The grid game uses the exact transition.

We use the following elementary Gaussian bounds. If $q_j(x,\xi)$ is a
polynomial of degree at most $j$ with uniformly bounded sufficiently
smooth coefficients, then, for $0<\tau\le1$,
\begin{align}
\label{eq:gaussian-symbol-bounds}
 \|\operatorname{Op}(q_j e^{\tau p})u\|_H
       &\le C\tau^{-j/2}\|u\|_H,\notag\\
 \|\operatorname{Op}(q_j e^{\tau p})u\|_J
       &\le C(1+\tau^{-(j-1)/2})\|u\|_H,
       \qquad j\ge1.
\end{align}
Indeed, a Gaussian differentiated $j$ times has kernel $L^1$ norm
$O(\tau^{-j/2})$. Derivatives of its coefficient parameters add factors
$\tau$ and at most two Gaussian derivatives, so they obey the same
bound. Differentiate the output up to the required integer order and
take H\"older differences. For the second estimate, transfer one
Gaussian derivative to $u$ by integration by parts. This also proves the
Banach-valued and action-differentiated versions. The argument uses only
Gaussian moments and remains uniform under convex interpolation of the
diffusion matrices.

\begin{lemma}[Exact transition versus local Gaussian]
\label{lem:gaussian-comparison}
For $0\le k\le4$,
\begin{equation}
\label{eq:gaussian-comparison}
 \|D_a^k(P_{t,t+h}^a-\Pi_{t,h}^a)u\|_H
       \le C\sqrt h\|u\|_H,
 \qquad
 \|D_a^k(P_{t,t+h}^a-\Pi_{t,h}^a)u\|_J
       \le Ch\|u\|_H.
\end{equation}
For $k=0,1$, the same bounds hold after substituting any uniformly
$C_b^{2+\beta}$ action $a=a(x)$.
\end{lemma}

\begin{proof}
With $\tau=t+h-s$, set $F(s)=P_{t,s}^a\Pi_{s,\tau}^a$.
Its endpoint values are $F(t)=\Pi_{t,h}^a$ and $F(t+h)=P_{t,t+h}^a$, since a transition over zero time is the identity.
Integrating $F'(s)$ therefore gives the exact identity
\[
 P_{t,t+h}^a-\Pi_{t,h}^a
 =\int_t^{t+h}P_{t,s}^a
                (\partial_s+L_s^a)\Pi_{s,t+h-s}^a\,\dd s.
\]
The residual $(\partial_s+L_s^a)\Pi_{s,t+h-s}^a$ measures the failure of the Gaussian comparison to satisfy the exact backward evolution equation.
Writing $p=p(s,x,a,\xi)$, its symbol is
\begin{equation}
\label{eq:gaussian-residual-symbol}
 e^{\tau p}\left[
 \tau\left(p_s+b\cdot p_x+i(c\xi)\cdot p_x
                         +\tfrac12 \operatorname{tr}(cp_{xx})\right)
          +\tfrac12\tau^2 \operatorname{tr}\!\left(c(p_x\otimes p_x)\right)\right].
\end{equation}
The term of order $\tau$ has polynomial degree at most three, and the
term of order $\tau^2$ degree at most four. Thus its $H\to H$ norm is
at most $C\tau^{-1/2}$ and its $H\to J$ norm at most $C$, by
\eqref{eq:gaussian-symbol-bounds}. Action differentiation preserves these
estimates: differentiating the exponential adds $\tau p_a$, of degree
at most two. Differentiating the integral also produces action
derivatives of $P_{t,s}^a$, bounded on $H$ and $J$ by
\eqref{eq:scaled-transition}.
Thus the strong-norm error is bounded by $C\int_0^h\tau^{-1/2}\,\dd\tau=2C\sqrt h$, whereas the weaker-norm error is bounded by $C\int_0^h1\,\dd\tau=Ch$.
This proves \eqref{eq:gaussian-comparison} for the fixed parameter.
For $k=0,1$, the chain rule and H\"older differences after substitution of any uniformly $C_b^{2+\beta}$ action use action derivatives only through order $k+3\le4$, proving the final statement.
In particular, the action is never moved inside a spatial semigroup.
\end{proof}

\paragraph{Reference linearization and its frozen model}
Let $a_i=\alpha_{i,h}[\bar U_{i+1}]$ and
$H_i=\bR_i\bM_i^{-1}$. Set
$p_i(x,\xi)=p(t_i,x,a_i(x),\xi),\quad q_i(x,\xi)=D_ap(t_i,x,a_i(x),\xi).$
For $\zeta=(\zeta_f,\zeta_g)$ in the fixed auxiliary range $E$, define
$w_i(x)\zeta=\zeta_f(x)+\gamma_i(x)\cdot\zeta_g$.
In the symbol, this bounded functional acts on the input's spatial Fourier
coefficient. Define
\begin{equation}
\label{eq:reference-gaussian-symbol}
 \mathcal E_{i,h}^{\rm G}
 =\operatorname{Op}\left(e^{hp_i(x,\xi)}
                   [I+hH_i(x)q_i(x,\xi)w_i(x)]\right).
\end{equation}
The coefficients $H_i,\gamma_i$ in this expression are the \emph{exact}
reference coefficients; they are not replaced by Gaussian approximations.
Since $D_a\Pi=h\operatorname{Op}(p_ae^{hp})$, the formula for $\bS$ and
\cref{lem:gaussian-comparison} show that
\begin{equation}
\label{eq:reference-gaussian-difference}
 D\cT_{i,h}(\bar U_{i+1})-\mathcal E_{i,h}^{\rm G}
                 =K_{i,h}^{(1)}+K_{i,h}^{(0)}
\end{equation}
with the bounds in \eqref{eq:two-norm-classes} and uniform $b_1,b_0$.
The $D_{22}^2G$ term is $K^{(0)}$: it has the form $h$ times a uniformly
$C_b^{2+\beta}$ coefficient times $P^{a_i}u_g$. All other differences use
\eqref{eq:gaussian-comparison}, including its first action derivative.

Fix a grid time $t_*$ in the block and a center $x_*$.
To apply the constant-coefficient estimates of \cref{lem:micro-duhamel}, freeze
$p,H,q,w$ in \eqref{eq:reference-gaussian-symbol} at $(t_*,x_*)$.
The resulting step is $\mathscr E_{*,h}=P_{*,h}(I+h\bN_{*,h})$.
The exact discrete first-order condition gives $w_*\bR_*=0$, hence
$\bN_{*,h}^2=0$ as in \cref{prop:nilpotent}.
No nilpotence is asserted for the intermediate symbols or for path averages.

\paragraph{Two-norm bounds for spatial freezing}
Let $\chi_*\in C_c^\infty(B_{2r}(x_*))$ equal one on $B_r(x_*)$.
Multiplication by $\chi_*$ localizes the starting point only; the transition still uses the full input function and may leave the ball.
The reference coefficients are uniformly $C_b^{2+\beta}$, and their
oscillation on the block and this ball is at most
$C(r^\beta+\Delta^{\beta/2})$. For the moving diagonal this follows from
\[
 \|E_x-E_{x_*}\|_{\mathcal L(C_b^{2+\beta}(\R_y^d),\R)}\le |x-x_*|,
 \qquad E_xv=v(x).
\]
In fact $x\mapsto E_x$ is $C_b^{2+\beta}$ in this operator norm.

Expand the difference of the two symbols in
\eqref{eq:reference-gaussian-symbol}, using
\[
 e^{hp_i}-e^{hp_*}
   =h(p_i-p_*)\int_0^1e^{h[(1-\theta)p_*+\theta p_i]}\,\dd\theta.
\]
Each resulting term, including parameter integrals, consists of a
coefficient difference $C_i(x)-C_*$ and a Gaussian multiplier $J_{i,h}$.
The latter is a finite sum of symbols $h$ times a polynomial of degree
one or two, and $h^2$ times a polynomial of degree at most four, all
times an interpolated Gaussian. Uniformly,
\begin{align}
\label{eq:gaussian-J-bounds}
 \|J_{i,h}u\|_H&\le C\|u\|_H,&
 \|J_{i,h}u\|_L&\le Ch\|u\|_H,\notag\\
 \|J_{i,h}u\|_0&\le C\|u\|_0,&
 \|J_{i,h}u\|_0&\le Ch\|u\|_{C_b^2}.
\end{align}
These estimates follow from the same Gaussian kernel argument, transferring
two derivatives to $u$ for the last bound and for $H\to L$.
Bounded coefficient factors can be retained in the multipliers; for
operator-valued differences the product estimate is read as an operator
acting on the Gaussian-smoothed vector.

The tame product inequality
\[
 \|Cv\|_{C_b^s}\le C_s(\|C\|_\infty\|v\|_{C_b^s}
                                  +\|C\|_{C_b^s}\|v\|_\infty)
\]
and \eqref{eq:gaussian-J-bounds} give, with
$\delta=r^\beta+\Delta^{\beta/2}$ and
$K_{i,h}=\chi_*(\mathcal E_{i,h}^{\rm G}-\mathscr E_{*,h})$,
\[
 \|K_{i,h}u\|_H\le C\delta\|u\|_H+C_r\|u\|_0,
 \qquad
 \|K_{i,h}u\|_L\le Ch\delta\|u\|_H+C_rh\|u\|_{C_b^2}.
\]
Here the small coefficient oscillation multiplies the strong input norm, while coefficient and cutoff derivatives enter the terms with weaker input norms.
Interpolation,
$\|u\|_{C_b^2}\le\varepsilon'\|u\|_H+C_{\varepsilon'}\|u\|_0$,
then yields
\begin{equation}
\label{eq:freezing-two-norm}
 \|K_{i,h}u\|_H\le b\|u\|_H+d\|u\|_0,
 \qquad
 \|K_{i,h}u\|_L\le h(b\|u\|_H+d\|u\|_0),
\end{equation}
where $b\le C(r^\beta+\Delta^{\beta/2})+\varepsilon$ and
$d\le C_{r,\varepsilon}$.
First fix $r$, then choose $\varepsilon'$ so that $C_r\varepsilon'\le\varepsilon$.
Thus large cutoff derivatives are absorbed into the coefficient of $\|u\|_0$, while the coefficient of $\|u\|_H$ remains small.

The commutator $\chi_*\mathscr E_{*,h}-\mathscr E_{*,h}\chi_*$ belongs to the first-order
class in \eqref{eq:two-norm-classes}, with constant $C_r$.
For completeness, its kernel is
$[\chi_*(x)-\chi_*(x-z)]k_h(z)$, where $k_h$ is the Gaussian kernel of
$P_{*,h}(I+h\bN_{*,h})$. Its absolute first moment is $O(\sqrt h)$,
which gives the $H\to H$ bound. To obtain the $H\to J$ bound, expand
the bracket as $D\chi_*(x)z+O_r(|z|^2)$. The second moment is $O(h)$.
The Fourier multiplier of $zk_h$ is the derivative in $\xi$ of
$e^{hp_*}(I+hn_*)$: it is a sum of $h$ times polynomials of degree at
most one and $h^2$ times polynomials of degree at most three, times
the Gaussian. Transferring one derivative to $u$ gives $Ch\|u\|_H$.
The Taylor remainder and its derivatives through order $1+\beta$ have
the same second-moment bound. This calculation retains the entire
factor $I+h\bN_{*,h}$.

\paragraph{Assembly and block closure}
For $W_i=A_{i,h}W_{i+1}$ and $Z_i^*=\chi_*W_i$, the exact recursion is
\[
 Z_i^*=\mathscr E_{*,h}Z_{i+1}^*+\mathcal R_{i,h}^*W_{i+1},
 \quad \mathcal R_{i,h}^*=
 \chi_*(A_{i,h}-\mathscr E_{*,h})+[\chi_*,\mathscr E_{*,h}].
\]
The remainder acts on the global input $W_{i+1}$.
The neighborhood and freezing errors in \cref{lem:scaled-tube} and \eqref{eq:freezing-two-norm} belong to $K^{(2)}$, the Gaussian comparison and cutoff commutator to $K^{(1)}$, and the $D_{22}^2G$ term to $K^{(0)}$; see \eqref{eq:reference-gaussian-difference}.
Cutoff multiplication and path averaging preserve the bounds, giving
\begin{equation}
\label{eq:localized-two-norm-sizes}
 b_2\le C(r^\beta+\Delta^{\beta/2}+\varepsilon)+C_r\rho,
 \qquad b_1+b_0\le C_r,\qquad d_2\le C_{r,\varepsilon}.
\end{equation}

Set $X=\max_i\|W_i\|_H$, $Y=\max_i\|W_i\|_0$, and $D=\|W_M\|_H$.
Repeated substitution in the localized recursion gives
\[
 Z_i^*=\mathscr E_{*,h}^{M-i}(\chi_*W_M)
       +\sum_{\ell=i}^{M-1}\mathscr E_{*,h}^{\ell-i}
                  \mathcal R_{\ell,h}^*W_{\ell+1}.
\]
Apply \cref{lem:micro-duhamel} to the global input sequence $W$: the three remainder classes contribute $C(b_2X+d_2Y)$, $Cb_1\sqrt\Delta X$, and $Cb_0\Delta X$, respectively.
The terminal contribution is at most $C_rD$ by boundedness of the frozen powers.
Using \eqref{eq:localized-two-norm-sizes}, we obtain
\[
 \max_i\|\chi_*W_i\|_H
 \le C_rD+
 [C(r^\beta+\Delta^{\beta/2}+\varepsilon)+C_r(\rho+\sqrt\Delta+\Delta)]X
 +C_{r,\varepsilon}Y.
\]
To recover the global norm, take the supremum over localization centers.
Pairs at distance less than $r$ lie in one inner localization ball; the contribution from more distant pairs is bounded by $Cr^{-\beta}\|D^2W\|_\infty$.
Interpolation bounds this last expression by $\theta\|W\|_H+C_{r,\theta}\|W\|_0$.
Fixing a small $\theta$ and absorbing its term yields
\[
 \|W\|_H\le C\sup_{x_*}\|\chi_*W\|_H+C_r\|W\|_0,
\]
where $C$ is independent of small $r$.
Consequently the local estimate gives the same global bound, with a uniform factor on its principal small terms and an additional $C_rY$.

Finally, the exact transition is a contraction in $Y_0$, and
$\|\bR\bM^{-1}\bS W\|_0\le C\|W\|_H$.
The exact linearization therefore gives, also after path averaging,
\[
 \|W_i\|_0\le\|W_{i+1}\|_0+Ch\|W_{i+1}\|_H.
\]
Summing over the block and using $Mh\le\Delta$ gives directly
\[
 Y\le\|W_M\|_0+Ch\sum_{\ell=0}^{M-1}\|W_{\ell+1}\|_H
   \le D+C\Delta X.
\]
Substitution into the global strong-norm estimate yields
\[
 X\le C_{r,\varepsilon}D+
 [C(r^\beta+\varepsilon+\Delta^{\beta/2})
       +C_r(\rho+\sqrt\Delta+\Delta)+C_{r,\varepsilon}\Delta]X.
\]
Choose $\varepsilon$ and $r$ small, then $\rho$ small depending on $r$,
and finally $\Delta_0$ small after all radius- and interpolation-dependent constants are fixed.
The bracket is at most $1/2$, independently of $h$.
Moving its contribution to the left gives $X\le2C_{r,\varepsilon}D$.
Since the terminal perturbation $W_M\in H$ is arbitrary and $W_i=A_{i,h}\cdots A_{M-1,h}W_M$, taking the supremum over $\|W_M\|_H=1$ gives the product-operator bound with $C_B=2C_{r,\varepsilon}$.
The same argument applies uniformly to every subblock of duration at most $\Delta_0$, proving \cref{prop:tube-stability}.

\section{Proof of one-step consistency}
\label{app:consistency-proof}

In this proof, $C_x^{2+\beta}C_a^4$ denotes uniform $C_b^{2+\beta}$ bounds in the state variable for the action derivatives through order four. For the continuation family $f$, the fixed $C_b^{2+\beta}$ regularity in the reference parameter $y$ is retained, as in $\cMix^{2+\beta,2+\beta}$.

\begin{proof}[Proof of \cref{prop:consistency}]
\smallskip\noindent\textit{Step 1: weak one-step expansion.}\enspace
Let $\phi=\phi(s,x)$ be sufficiently smooth and define
$\mathscr P_h^a\phi(t,x) :=P_{t,t+h}^a[\phi(t+h,\cdot)](x).$
The time-dependent Dynkin formula accounts for both the time change of $\phi$ and the state evolution, and gives
\begin{equation}
\label{eq:weak-first}
  \mathscr P_h^a\phi(t)-\phi(t)
  =\int_0^hP_{t,t+s}^a[\mathcal A_{t+s}^a\phi(t+s)]\,\dd s.
\end{equation}
Set $\psi^a(s):=\mathcal A_s^a\phi(s)$.
Applying the same identity to this integrand gives
\[
 P_{t,t+s}^a\psi^a(t+s)
 =\psi^a(t)+\int_0^sP_{t,t+r}^a
       [(\partial_r+L_{t+r}^a)\psi^a(t+r)]\,\dd r.
\]
Substitution into \eqref{eq:weak-first} yields
\begin{equation}
\label{eq:weak-expansion}
  P_{t,t+h}^a\phi(t+h)
  =\phi(t)+h\mathcal A_t^a\phi(t)
   +h^2\Rem_h^a[\phi](t),
\end{equation}
where
\begin{equation*}
  \Rem_h^a[\phi](t)
  :=\frac1{h^2}\int_0^h\int_0^s
  P_{t,t+r}^a[(\partial_r+L_{t+r}^a)\psi^a(t+r)]\,\dd r\,\dd s.
\end{equation*}
The factor $h^2$ corresponds to the area $h^2/2$ of the time triangle $0<r<s<h$.
The integrand is $(\partial_r+L_{t+r}^a)^2\phi$; since each generator has order two, it uses at most four state derivatives of $\phi$.
Bounding it in $C_b^{2+\beta}$ therefore requires the $6+\beta$ state regularity in \cref{ass:H5}.
Action derivatives act on the coefficients and transition operators without increasing this state-derivative requirement.
Differentiating the integral up to four times in the action and applying \eqref{eq:action-zero-order} therefore gives, under \cref{ass:H2,ass:H4,ass:H5},
$\sup_{a\in A} \|\Rem_h^a[\phi](t)\|_{C_x^{2+\beta}C_a^4} \le C.$

The affine function $\iota(x)=x$ is treated separately.  Since $L_s^a\iota=b(s,\cdot,a)$, Dynkin's formula and one further application of \eqref{eq:weak-expansion} give
\begin{equation}
\label{eq:id-weak-expansion}
  P_{t,t+h}^a\iota
  =\iota+hb(t,\cdot,a)+O_{C_x^{2+\beta}C_a^4}(h^2).
\end{equation}
Combining \eqref{eq:weak-expansion} for $\tgb$ with \eqref{eq:id-weak-expansion} yields the same weak expansion for $g=\Id+\tgb$.

\smallskip\noindent\textit{Step 2: expansion of the one-step objective.}\enspace
Apply \eqref{eq:weak-expansion} to $f^x$ and $g$, and Taylor-expand
$G$ in its second argument. With remainders uniformly bounded in
$C_x^{2+\beta}C_a^4$, this gives
\begin{equation}
\label{eq:qhat-expansion}
 \widehat q_{t,h}(x,a):=\frac{\overline Q_h(t,x,a)-V(t,x)}h
       =\cH(t,x,a)+h\Rem_{Q,h}(t,x,a).
\end{equation}
The two spatial and four action derivatives of the Taylor remainder
require at most the eight derivatives of $G$ in \cref{ass:H4}.

\smallskip\noindent\textit{Step 3: persistence of strong concavity and the interior margin.}\enspace
By \cref{ass:H7} and \eqref{eq:qhat-expansion}, for all sufficiently small $h$,
\begin{equation}
\label{eq:qhat-concavity}
  -D_{aa}^2\widehat q_{t,h}(x,a)
  \ge\frac{\lambda_H}{2}I_m
\end{equation}
uniformly in $(t,x,a)$.  Thus $a\mapsto\overline Q_h(t,x,a)$ is globally strongly concave on the convex compact set $A$ and has a unique global maximizer.

Write $\widehat a=\whalpha(t,x)$. Strong concavity of $\cH$ and optimality of $a_h^\star$ for $\widehat q_{t,h}=\cH+h\Rem_{Q,h}$ give
\[
 \frac{\lambda_H}{2}|a_h^\star-\widehat a|^2
 \le\cH(t,x,\widehat a)-\cH(t,x,a_h^\star)
 \le2h\|\Rem_{Q,h}\|_\infty.
\]
Thus $|a_h^\star-\widehat a|\le C\sqrt h$, without assuming interiority of $a_h^\star$.
This preliminary estimate and the margin in \cref{ass:H8} imply, after decreasing $h_0$ so that $C\sqrt h<\eta/2$,
$\operatorname{dist}(a_h^\star(t,x),\partial A)\ge\frac\eta2.$
Hence the first-order condition $D_a\widehat q_{t,h}(x,a_h^\star(t,x))=0$ is valid.

\smallskip\noindent\textit{Step 4: optimizer consistency.}\enspace
To obtain the $C_b^{2+\beta}$ estimate, connect the two objectives by
\[
 q_\theta=(1-\theta)\cH+\theta\widehat q_{t,h}
          =\cH+\theta h\Rem_{Q,h},\qquad 0\le\theta\le1,
\]
and let $a_\theta(x)$ be its maximizer.
Then $a_0=\whalpha$ and $a_1=a_h^\star$; we estimate their difference by integrating $\partial_\theta a_\theta$.
Step~3 applies uniformly in $\theta$, so these maximizers are interior and their negative Hessians are bounded below by $\lambda_H/2$.

The link to \cref{app:sensitivity} is the normalized bound of \cref{lem:S} with $r=2$:
\[
 \frac1h\|D_a^jP_{t,t+h}^a\varphi\|_{C_b^{2+\beta}}
 \le C\|\varphi\|_{C_b^{4+\beta}},\qquad 1\le j\le5.
\]
The reference data in \cref{ass:H5} have this regularity; use \cref{cor:aux-param} for the mixed norms of $f$ and \cref{cor:affine-data} for the identity part of $g$.
For $j\ge1$, $D_a^j\widehat q_{t,h}=h^{-1}D_a^j\overline Q_h$, since $V$ is independent of $a$.
Each differentiated composition with $G$ contains at least one action-sensitive transition factor of size $O(h)$ in $C_b^{2+\beta}$, canceling this normalization.
Thus $D_a^j\widehat q_{t,h}$, and hence $D_a^jq_\theta$, are uniformly bounded in $C_x^{2+\beta}$ for $1\le j\le5$.
These higher-order bounds are obtained directly from the transition estimates, independently of the four-derivative remainder estimate in Step~2.

Differentiating $D_aq_\theta(x,a_\theta(x))=0$ twice in $x$ and taking H\"older differences gives uniform $C_b^{2+\beta}$ bounds for $a_\theta$ and
$M_\theta(x):=-D_{aa}^2q_\theta(x,a_\theta(x))$.
For the substituted Hessian, two action derivatives are already present, two more may arise from the spatial chain rule, and one further derivative controls the H\"older differences; this explains the bound through order five.
Uniform positive definiteness gives the same spatial regularity for $M_\theta^{-1}$.
Differentiating in $\theta$ yields
\[
 \partial_\theta a_\theta
   =hM_\theta^{-1}D_a\Rem_{Q,h}(x,a_\theta(x)).
\]
For the substituted remainder gradient, the analogous count is $1+2+1=4$ action derivatives, supplied by Step~2.
The product and composition estimates therefore give
\[
 \|a_h^\star-\whalpha\|_{C_b^{2+\beta}}
 \le\int_0^1\|\partial_\theta a_\theta\|_{C_b^{2+\beta}}\,\dd\theta
 \le Ch,
\]
proving \eqref{eq:a-star-rate} without additional regularity assumptions.

\smallskip\noindent\textit{Step 5: continuation residual.}\enspace
For $\phi=f(\cdot,\cdot;y)$ or $g$, the weak expansion,
continuation equation and \eqref{eq:a-star-rate} give
\[
 P_{t,t+h}^{a_h^\star}\phi(t+h)-\phi(t)
 =h(L_t^{a_h^\star}-L_t^{\whalpha})\phi(t)+O(h^2)=O(h^2).
\]
For $f$, this holds in $\cMix^{2+\beta,2+\beta}$ uniformly in the
frozen parameter. For $g=\Id+\tgb$, use \eqref{eq:id-weak-expansion}
and $Dg=I+D\tgb$, $D^2g=D^2\tgb$ to bound the residual in
$C_b^{2+\beta}$.
In both component norms, coefficient smoothness and \eqref{eq:a-star-rate} bound the generator difference by $Ch$; its prefactor $h$ therefore gives $O(h^2)$, as does the weak-expansion remainder.
The identity part cancels in the recentered second component.
Thus the definition of $\cT_{t,h}$ and the sum norm on $\cX^{2+\beta}$ give
\[
 \begin{aligned}
 \|\cT_{t,h}(U(t+h))-U(t)\|_{\cX^{2+\beta}}
 ={}&\|P_{t,t+h}^{a_h^\star}f(t+h)-f(t)\|_{\cMix^{2+\beta,2+\beta}}\\
 &+\|P_{t,t+h}^{a_h^\star}g(t+h)-g(t)\|_{C_b^{2+\beta}}
 \le Ch^2.
 \end{aligned}
\]
This is precisely the local residual estimate \eqref{eq:consistency-residual}.
\end{proof}

\section*{Declaration of AI use}
OpenAI ChatGPT was used for proof exploration and checking, literature
searches, and drafting and revising text and \LaTeX.
The author has independently verified all mathematical arguments and read all cited works.
The author assumes responsibility for all content.

\end{document}